\documentclass[reqno,12pt]{amsart}
\usepackage{amsmath,amsfonts,amssymb,amsthm}
\usepackage{etoolbox}
\usepackage{latexsym}
\usepackage{mathrsfs}
\usepackage{bm,thmtools} %% SAM

\usepackage{esint}
\newcommand{\wstarto}{\overset{\ast}{\rightharpoonup}}

\newcommand{\T}{\mathbb{T}}
\newcommand{\R}{\mathbb{R}}

\newcommand{\N}{\mathbb{N}}

\newcommand{\Z}{\mathbb{Z}}

\newcommand{\cM}{\mathcal{M}}

\newcommand{\define}{\textbf}

\renewcommand{\tilde}{\widetilde}

\renewcommand{\bar}{\overline}

\newcommand{\embeds}{{\hookrightarrow}}
\newcommand{\weaklyto}{{\rightharpoonup}}

\newcommand{\set}[2]{{\left\{ #1 \,\middle|\, #2 \right\}}}
\newcommand{\slot}{{\,\cdot\,}}

\let\div\divv
\newcommand{\div}{{\mathrm{div}}}
\usepackage[pagebackref=true, colorlinks=true, citecolor=blue, pdfborder={0 0 .1},breaklinks]{hyperref}

\usepackage[initials,nobysame]{amsrefs}% Load after hyperref
\usepackage{enumitem}

\title[Energy balance]{Sharp conditions for energy balance in two-dimensional incompressible ideal flow with external force}

\author[Jin]{Fabian Jin$^1$}
\author[Lanthaler]{Samuel Lanthaler$^2$}
\author[Lopes Filho]{Milton C. Lopes Filho$^3$}
\author[Nussenzveig Lopes]{Helena J. Nussenzveig Lopes$^3$}
\address{$^1$ Seminar for applied mathematics\\ETH Zurich\\Raemistrasse 101, 8092 Zurich, Switzerland}
\email{fabian.jin@sam.math.ethz.ch}
\address{$^2$ Computing and Mathematical Sciences\\California Institute of Technology\\1200 East California Boulevard,
Pasadena, CA--91125, USA}
\email{slanth@caltech.edu}
\address{$^3$ Instituto de Matem\'atica\\Universidade Federal do Rio de Janeiro\\Cidade Universit\'aria -- Ilha do Fund\~ao\\Caixa Postal 68530\\21941-909 Rio de Janeiro, RJ -- BRAZIL\\}
\email{mlopes@im.ufrj.br}
\email{hlopes@im.ufrj.br}

\declaretheoremstyle[
  headfont=\normalfont\bfseries\itshape,
  numbered=unless unique,
  bodyfont=\normalfont,
  spaceabove=1em plus 0.75em minus 0.25em,
  spacebelow=1em plus 0.75em minus 0.25em,
  qed={},
]{deflt}

\theoremstyle{deflt}
\newtheorem{theorem}{Theorem}[section]

\newtheorem{example}[theorem]{Example}
\newtheorem{remark}[theorem]{Remark}

\newtheorem{definition}[theorem]{Definition}
\newtheorem{lemma}[theorem]{Lemma}
\newtheorem{proposition}[theorem]{Proposition}
\newtheorem{corollary}[theorem]{Corollary}

\numberwithin{equation}{section}
\numberwithin{theorem}{section}

\newcommand{\vertiii}[1]{{\left\vert\kern-0.25ex\left\vert\kern-0.25ex\left\vert #1
    \right\vert\kern-0.25ex\right\vert\kern-0.25ex\right\vert}}

\renewcommand{\tilde}{\widetilde}

\DeclareMathOperator{\dv}{div}
\DeclareMathOperator{\curl}{curl}

\begin{document}

\begin{abstract}
Smooth solutions of the forced incompressible Euler equations satisfy an energy balance, where the rate-of-change in time of the kinetic energy equals the work done by the force per unit time. Interesting phenomena such as turbulence are closely linked with rough solutions which may exhibit {\it inviscid dissipation}, or, in other words, for which energy balance does not hold.
This article provides a characterization of energy balance for physically realizable weak solutions of the forced incompressible Euler equations, i.e. solutions which are obtained in the limit of vanishing viscosity. More precisely, we show that, in the two-dimensional periodic setting, strong convergence of the zero-viscosity limit is both necessary and sufficient for energy balance of the limiting solution, under suitable conditions on the external force. As a consequence, we prove energy balance for a general class of solutions with initial  vorticity belonging to rearrangement-invariant spaces, and going beyond Onsager's critical regularity.
\end{abstract}

\maketitle

\section{Introduction}

The present work concerns weak solutions of the forced incompressible Euler equations on a two-dimensional periodic domain, focusing on solutions that arise as vanishing viscosity limits.
The Euler equations describe the motion of an idealized fluid in the absence of friction and other diffusive effects, and formally satisfy the principle of conservation of energy \cite{MajdaBertozzi}. Without an external force, conservation of energy is the conservation in time of the square of the $L^2$-norm of the fluid velocity. When allowing for an external force in the equations of motion, the conservation of energy instead takes the form of an identity in which the rate-of-change of the kinetic energy equals the work of the external force per unit time \cite{LN2022}.

While smooth solutions conserve energy, this may not be the case for weak solutions. In fact, weak solutions which dissipate energy are an essential part of Kolmogorov's K41 theory of homogeneous 3D turbulence. In 1949, L. Onsager argued that only ideal flows with a third of a derivative or less may exhibit inviscid dissipation, a statement later referred to as the {\it Onsager Conjecture}, see \cite{Kolmogorov41a,Kolmogorov41b,Frisch1995,ES2006}. There is substantial recent work connected with this issue. Inviscid dissipation is indeed ruled out for weak solutions at higher than $1/3$ regularity, see \cite{CET1994,CCFS08}. Starting with the pioneering work of DeLellis and Székelyhidi \cite{LS2009}, convex integration techniques have been employed to rigorously prove that dissipative Hölder continuous solutions exist in three spatial dimensions \cite{DS2017,Isett2018,BLSV2019}, up to the Onsager-critical regularity. Despite this extensive body of work, it is unclear whether such weak solutions can be obtained in the zero-viscosity limit, and hence their physical significance remains unclear.

For solutions obtained in the zero-viscosity limit in two dimensions, Cheskidov, Nussenzveig Lopes, Lopes Filho and Shvydkoy \cite{CLNS2016} prove \emph{energy conservation} under the assumption that the initial vorticity is $p$-th power integrable, for $p>1$, and in the absence of forcing. This is surprising in view of the fact that for $p<3/2$, such solutions go beyond Onsager-criticality (see \cite{Shvydkoy2010} for a discussion of Onsager-criticality), and their result hints at non-trivial dynamic constraints for solutions obtained in the zero-viscosity limit. Following \cite{CLNS2016}, solutions obtained in the zero-viscosity limit will be hereafter referred to as ``physically realizable''.

Still in the absence of forcing, the results of \cite{CLNS2016} have subsequently been extended to provide a complete characterization of physically realizable energy-conservative solutions in \cite{LMPP2021a}, where it is shown that energy conservation of such solutions is equivalent to the strong convergence of the zero-viscosity limit. In a different direction, Ciampa \cite{ciampa2022energy} derives sufficient conditions for energy conservation when the fluid domain is the full plane. The forced periodic case has been considered in \cite{LN2022}, where sufficient conditions for energy balance of physically realizable solutions were derived based on suitable $L^p$-integrability assumptions on the vorticity and the curl of the external force.

External forcing is a natural mechanism for the generation of small scales in incompressible flow, and small-scale motion is an inherent feature of turbulence. Hence, finding precise conditions which rule out inviscid dissipation in the presence of an external force is particularly pertinent. In view of the characterization of physically realizable energy conservative solutions in the unforced case \cite{LMPP2021a}, one might hope that a similar characterization could also be obtained for energy balanced solutions when considering external forcing. This is the motivation for the present work.

We derive necessary and sufficient conditions for energy balance in the zero-viscosity limit for solutions of the two-dimensional Euler equations with external force. Under suitable assumptions on the force, we show that energy balance in the limit is equivalent to the strong convergence in the zero-viscosity limit. We then discuss how the results obtained here sharpen those of the previous work \cite{LN2022}, and allow for an extension of those results to vorticity in general rearrangement-invariant spaces with compact embedding in $H^{-1}$.
To prove such an extension, we derive novel {\it a priori} bounds for the rearrangement-invariant maximal vorticity function for solutions of the Navier-Stokes equations. To the best of our knowledge such a result is not contained in the literature on the two-dimensional Navier-Stokes equations. Our work also fills a gap in \cite[Corollary 2.13]{LMPP2021a}, where these bounds were asserted without proof.

We briefly discuss two results related to our work. First, in \cite{BD2022}, E. Bru\`{e} and C. De Lellis construct a sequence of viscous approximations, converging weakly to a solution of the incompressible Euler equations on the three-dimensional torus $\T^3$, exhibiting both anomalous dissipation along the sequence, see Definition \ref{def:anomdissinvdiss}, and inviscid dissipation. This is an illustration of the sharpness of our result, in particular since the implication ``energy balance $\Longrightarrow$ strong convergence" holds in any spatial dimension, see Remark \ref{rem:anydimension}. Second, in \cite{Cheskidov23}, A. Cheskidov, carries out an extensive discussion of the vanishing viscosity limit in the three-dimensional torus, constructing a vanishing viscosity sequence which exhibits both anomalous and inviscid dissipation, and, surprisingly, an example of inviscid dissipation without anomalous dissipation along the viscous approximation. All examples mentioned above require forcing. For details, and additional related work, see \cite{BD2022,Cheskidov23} and references therein.

\emph{Overview:} In Section \ref{sec:probsetting}, we recall several definitions before stating our main result, Theorem \ref{thm:ebal}, that strong convergence in the zero-viscosity limit is both necessary and sufficient for energy conservation of physically realizable solutions. The proof of this theorem is detailed in Section \ref{sec:ebal};  necessity is shown in subsection \ref{sec:ebal-->conv}, sufficiency in subsection \ref{sec:conv-->ebal}. Applications of our main result to rearrangement-invariant spaces can be found in Section \ref{sec:examples}; we refer to subsection \ref{sec:Lp} for a discussion of $L^p$-bounded vorticity, subsection \ref{sec:rinv} for {\it a priori} estimates in more general rearrangement-invariant spaces, and subsection \ref{sec:Lorentz} for sufficient conditions for energy balance of solutions with vorticity in the Lorentz space $L^{(1,2)}$, the largest rearrangement-invariant space with continuous embedding in $H^{-1}$ (cf. Theorem \ref{thm:L12}).
The derivation of our {\it a priori} estimates in rearrangement-invariant spaces is based on operator splitting. We include in Appendix \ref{app:adv-diff} the required results on convergence of these particular operator splitting approximations.
%***CHECK APPENDIX AND POSSIBLY RE-WRITE***

\section{Problem setting}
\label{sec:probsetting}

We study the incompressible Euler equations with initial data $u_0$, external forcing $f$ and subject to periodic boundary conditions, given by
\begin{gather}
\label{eq:euler}
\left\{
\begin{array}{ll}
  \partial_t u + u\cdot \nabla u + \nabla p = f, \quad &\text{in } \T^2 \times (0,T), \\
  \dv  u = 0, &\text{in } \T^2 \times [0,T], \\
  u = u_0, & \text{at } \T^2 \times \{0\}.
\end{array}
\right.
\end{gather}
Above $\T^2$ denotes the two-dimensional flat torus. Throughout this work, we fix a time-horizon $T>0$. We also assume, without loss of generality, that the forcing $f$ is divergence-free as, otherwise, the gradient part can be absorbed in the pressure term. In the following, we are interested in the evolution of the kinetic energy $\frac12 \Vert u(t) \Vert_{L^2_x}^2$ of \emph{weak solutions} of the system \eqref{eq:euler}.

\begin{definition}
  \label{def:weaksol}
  Fix $T>0$, let $u_0 \in L^2 (\T^2)$ be a divergence-free vector field. Assume $f\in
  L^1 (0,T;L^2 (\T^2))$, $\dv f = 0$. A vector field $u\in L^\infty (0,T;L^2 (\T^2))$ is a \define{weak solution} of \eqref{eq:euler}, if
    \begin{enumerate}
      \item for all divergence-free test vector fields $\phi\in C_c^\infty(\T^2 \times [0,T))$, we have
\begin{multline}
\int_0^T
\int_{\T^2} \left\{ u\cdot \partial_t \phi
+ u\otimes u : \nabla \phi \right\} \, dx \, dt
+ \int_{\T^2}  u_0(x) \cdot \phi(x,0) \, dx \\
= \int_0^T \int_{\T^2} f \cdot \phi \, dx \, dt,
\end{multline}
\item for almost every $t\in [0,T]$, $\dv  u(\slot,t) = 0$ holds in the sense of distributions.
  \end{enumerate}
\end{definition}

Existence of weak solutions in the sense of Definition \ref{def:weaksol} can be established under additional assumptions, such as $L^p$-bounds, $p\geq 1$, on the initial vorticity $\omega_0 = \curl(u_0)$ and on the curl of the forcing \cite{LN2022}.

Before moving forward let us comment on notation. We will use the subscript `c' in a function space to denote elements of the space whose support is compact; we use, whenever convenient, subscripts `t' and `x' to denote time and spatial dependence respectively, so that $L^2_tH^1_x$ is shorthand for $L^2((0,T);H^1(\T^2))$.

Given that the Euler equations describe the motion of an ideal fluid, i.e. one for which friction and other dissipative effects are neglected, it is reasonable to expect that a change in kinetic energy can only be caused by the action of the external forcing $f$. Following \cite{LN2022}, we recall the following definition.

\begin{definition}
  \label{def:ebal}
  We say $u\in L^\infty((0,T);L^2 (\T^2))$ is an \define{energy balanced weak solution} if $u$ is a weak solution of the incompressible Euler equations with forcing $f\in L^1((0,T);L^2(\T^2))$, $\dv f = 0$,  such that
    \begin{align}
      \label{eq:ebal}
      \frac12 \Vert u(t) \Vert^2_{L^2_x}
      =
      \frac12 \Vert u_0 \Vert^2_{L^2_x}
      +
      \int_0^t \langle f, u \rangle_{L^2_x} \, d\tau,
    \end{align}
    for almost every $t \in [0,T]$.
    Here, $\langle f, u \rangle_{L^2_x}(t) := \int_{\T^2} f(x,t) \cdot u(x,t) \, dx$ denotes the $L^2_x$-inner product.
\end{definition}

\begin{remark}
\label{rem:strongtcont}
 Note that the right-hand-side of \eqref{eq:ebal} is continuous with respect to $t$. Therefore, redefining $u$ on a set of measure zero in $[0,T]$, we will assume hereafter that an energy balanced weak solution $u$ satisfies \eqref{eq:ebal} for all $t \in [0,T]$ and, hence, $t \mapsto \|u(t)\|_{L^2}$ is continuous.

  %We will show below (cf. Lemma \ref{lem:tcont}) that all solutions of relevance to the present work are continuous functions of time with values in weak-$L^2_x$. For such solutions $u\in C([0,T];w\text{-}L^2(\T^2))$, the energy balance relation \eqref{eq:ebal} implies continuity of the norm $t \mapsto  \Vert u(t) \Vert_{L^2_x}$, which in turn implies that $u \in C([0,T];L^2_x)$. Hence the assumption that $u \in C([0,T];L^2_x)$ in definition \ref{def:ebal} is automatically satisfied in this case.
\end{remark}

In the absence of additional constraints, weak solutions with initial data $u_0\in L^2_x$ are not unique and may not satisfy energy balance (or even energy conservation, without external forcing), see  \cite{szekelyhidi2011weak} for an example with vortex sheet data. It is thus natural to impose additional constraints. Such constraints arise, for example, when considering the Euler equations \eqref{eq:euler} as the zero-viscosity limit of the physically relevant Navier-Stokes equations:
\begin{gather}
\label{eq:ns}
\left\{
\begin{array}{ll}
  \partial_t u^\nu + u^\nu \cdot \nabla u^\nu + \nabla p^\nu = \nu \Delta u^\nu +  f^\nu,
  \quad &\text{in }\T^2\times (0,T), \\
  \dv(u^\nu) = 0, & \text{in } \T^2 \times [0,T)\\
  u^\nu = u^\nu_0, & \text{at } \T^2 \times \{0\}.
\end{array}
\right.
\end{gather}

In contrast to the incompressible Euler equations, it is well-known that the initial value problem \eqref{eq:ns} is well-posed for $u^\nu_0\in L^2(\T^2)$, $f^\nu \in L^1 ((0,T);L^2 (\T^2))$, and the  solution $u^\nu$ belongs to $ L^\infty ((0,T);L^2 (\T^2)) \cap L^2 ((0,T);H^1 (\T^2))$, e.g. \cite{Lions} and references therein. Furthermore, since we consider only two dimensional flows, solutions of \eqref{eq:ns} satisfy the following energy identity for all $t\in [0,T]$,
\begin{align}
\label{eq:ns-ebal}
  \frac12\Vert u^\nu(t) \Vert^2_{L^2_x}
  &=
  \frac12\Vert u_0^\nu \Vert_{L^2_x}^2
  -\nu \int_0^t \Vert \omega^\nu(\tau) \Vert^2_{L^2_x} \, d\tau
  + \int_0^t \langle f^\nu, u^\nu \rangle_{L^2_x} \, d\tau .
\end{align}
 Above, $\omega^\nu = \curl(u^\nu)$ denotes the vorticity of the flow $u^\nu$. Formally, the \emph{energy dissipation} term $\nu \int_0^t \Vert \omega^\nu(\tau) \Vert_{L^2_x}^2 \, d\tau$ vanishes when $\nu = 0$, corresponding to the energy balance relation \eqref{eq:ebal}.

 \begin{remark} \label{viscencont}
 Observe that, since $u^\nu \in L^\infty ((0,T);L^2 (\T^2)) \cap L^2 ((0,T);H^1 (\T^2))$ it follows immediately from \eqref{eq:ns-ebal} that $t \mapsto \| u^\nu(t)\|_{L^2_x}$ is continuous on $[0,T]$. \end{remark}

Let us recall the definition of \emph{physically realizable} solutions of \eqref{eq:euler}:
\begin{definition} \label{def:physrealwsoln}
A weak solution $u \in L^\infty ((0,T);L^2 (\T^2))$ of the incompressible Euler equations with (divergence-free) forcing $f\in L^1 ((0,T);L^2 (\T^2))$ is \define{physically realizable}, if there exists a sequence $\nu \to 0$, initial data $u_0^\nu \in L^2 (\T^2)$ and forces $f^\nu \in L^1 ((0,T); L^2 (\T^2))$, $\dv f^\nu = 0$, such that
\begin{itemize}
\item $u_0^\nu \to u_0$ strongly in $L^2 (\T^2)$,
\item $f^\nu \weaklyto f$ weakly in $L^1 ((0,T); L^2 (\T^2))$,
\end{itemize}
and the corresponding solutions $u^\nu$ of the Navier-Stokes equations \eqref{eq:ns}
\begin{itemize}
\item $u^\nu \weaklyto u$ weakly-$\ast$ in $L^\infty ((0,T);L^2  (\T^2))$.
  \end{itemize}
In this case, the sequence $u^\nu$, $\nu \to 0$, is referred to as a \define{physical realization} of $u$.
\end{definition}

The class of physically realizable solutions was originally introduced in \cite{CLNS2016} in the absence of forcing, and extended to the forced case in \cite{LN2022}. In both of these papers sufficient conditions for physically realizable solutions to be energy conservative, or energy balanced, were obtained in terms of $L^p$-control of the vorticity. For the unforced case, a sharp characterization of energy conservation for physically realizable solutions was achieved in \cite{LMPP2021a}. The goal of the present work is to carry out a programme similar to \cite{LMPP2021a} in the forced case.

\begin{definition} \label{def:anomdissinvdiss}
Let $u$ be a physically realizable weak solution and consider $u^\nu \rightharpoonup u$ a physical realization. Let $\omega^\nu \equiv \curl (u^\nu)$. We say the family $\{u^\nu\}_\nu$ exhibits {\em anomalous dissipation} if
\[\liminf_{\nu \to 0^+} \nu \int_0^t \Vert \omega^\nu(\tau) \Vert_{L^2_x}^2 \, d\tau > 0.\]
\end{definition}

\begin{remark} \label{rem:anomdissinvdiss}
Note that our definition of {\em anomalous dissipation} corresponds to what was defined as ``dissipation anomaly" in \cite{Cheskidov23}. Furthermore, what was defined as ``anomalous dissipation" in \cite{Cheskidov23} is what we refer to as {\em inviscid dissipation}.
\end{remark}

We are now ready to state our main result.

\begin{theorem}
  \label{thm:ebal}
Let $u$ be a physically realizable solution of the incompressible Euler equations \eqref{eq:euler} with divergence-free forcing $f\in L^2 ((0,T);L^2 (\T^2))$ and initial data $u_0\in L^2(\T^2)$. Let $u^\nu$, $\nu \to 0$, be a physical realization of $u$ with forcing $f^\nu \in L^2 ((0,T);L^2 (\T^2))$, $\dv f^\nu = 0$, and initial data $u^\nu_0 \in L^2 (\T^2)$. Assume that the convergence $f^\nu \to f$ is strong in
$L^2 ((0,T);L^2 (\T^2))$. Then the following assertions are equivalent:
\begin{enumerate}
\item $u$ is energy balanced, \label{item1}
\item the convergence $u^\nu \to u$ is strong in $L^2 ((0,T);L^2 (\T^2))$, \label{item2}
\item the convergence $u^\nu \to u$ is strong in $C([0,T];L^2 (\T^2))$. \label{item3}
\end{enumerate}
\end{theorem}

Before delving into the proof of Theorem \ref{thm:ebal} we make a few remarks on the assumptions.

\begin{remark}
   Observe that, from the hypotheses of Theorem \ref{thm:ebal}, we assume implicitly that the sequence of external forcings $f^\nu$ is uniformly bounded in $L^2_tL^2_x$ as $\nu \to 0$. This assumption is related to vorticity estimates. Indeed, recall the vorticity formulation of the the Navier-Stokes equations \eqref{eq:ns}:
  \begin{align} \label{nuvorteq}
    \partial_t \omega^\nu + u^\nu \cdot \nabla \omega^\nu = \nu \Delta \omega^\nu + g^\nu,
  \end{align}
  with $g^\nu = \curl(f^\nu)$. Multiplying \eqref{nuvorteq} by $\omega^\nu$, and integrating the forcing term by parts gives
  \begin{align*}
  \frac{d}{dt} \frac12 \Vert \omega^\nu \Vert_{L^2_x}^2
  &=
  - \nu \Vert \nabla \omega^\nu \Vert_{L^2_x}^2 + \int_{\T^2} g^\nu \omega^\nu \, dx
  \\
  &=
  - \nu \Vert \nabla \omega^\nu \Vert_{L^2_x}^2 - \int_{\T^2}  f^\nu \cdot \nabla^\perp\omega^\nu \, dx.
  \end{align*}
  We estimate the last term from above by $ \frac12\nu \Vert \nabla \omega^\nu \Vert_{L^2_x}^2 + (2\nu)^{-1} \Vert f^\nu \Vert_{L^2_x}^2$, thus obtaining the differential inequality
  \begin{align}
    \frac{d}{dt} \Vert \omega^\nu \Vert_{L^2_x}^2 \le - \nu \Vert \nabla \omega^\nu \Vert^2_{L^2_x} +  \frac{1}{\nu} \Vert f^\nu \Vert_{L^2_x}^2.
    \end{align}
    Neglecting the non-positive term and upon integration in time, we find
    \begin{align}
      \label{eq:enst}
    \Vert \omega^\nu (t) \Vert^2_{L^2}
    \le
    \Vert \omega^\nu(\tau) \Vert^2_{L^2} + \frac1{\nu} \int_\tau^t \Vert f^\nu \Vert^2_{L^2_x} \, ds,
    \end{align}
    for $\tau\in (0,t]$. We will see that, in order to obtain a bound on $\Vert \omega^\nu(t) \Vert^2_{L^2}$, we require $f^\nu \in L^2_tL^2_x$; see Lemma \ref{lem:vort-bd}.
\end{remark}

\begin{remark}
  In Theorem \ref{thm:ebal} it is furthermore assumed that $f^\nu $ converges strongly to $f$ in $L^2_tL^2_x$ as $\nu \to 0$. This assumption ensures that
  \begin{equation} \label{convforcterm}
  \int_0^T \langle f^\nu, u^\nu \rangle_{L^2_x} \, dt \to \int_0^T \langle f, u \rangle_{L^2_x} \, dt,
   \end{equation}
   thus ruling out the failure of energy balance arising from the forcing term. If the convergence $u^\nu \to u$ is strong, then \eqref{convforcterm} holds even under the relaxed condition that $f^\nu \weaklyto f$ only weakly in $L^2_tL^2_x$. In fact we will see in Proposition \ref{prop:conv-->ebal}, that, assuming strong convergence of $u^\nu$ to $u$, energy balance of the limit $u$ follows under this weaker condition on $f^\nu$, $f$.
\end{remark}

   In Example \ref{ex:ebal} below we will show that the hypothesis of strong convergence of $f^\nu$ to $f$ is actually required for Theorem \ref{thm:ebal}.
   More precisely we exhibit a physically realizable solution $u$ which is energy balanced, for which the forcings $f^\nu$ of the physical realization converge weakly in $L^2_t L^2_x$ to the forcing $f$ of the limit $u$, yet $u^\nu$ does \emph{not converge strongly} to $u$. Therefore the direction ``energy balance $\Rightarrow$ strong convergence'' does not hold under the assumption that $f^\nu$ converges only weakly to $f$.

\begin{example}
  \label{ex:ebal}
  We claim that $u\equiv 0$ is a physically realizable solution, for which there exists a physical realization $u^\nu \weaklyto u$ with forcing $f^\nu \weaklyto 0$ in $L^2_tL^2_x$, but such that $u^\nu $ does not converge strongly to $u$ in $L^2_t L^2_x$.

  To this end, fix two non-zero functions $\gamma, \phi \in C^\infty_c((0,+\infty))$ supported on $[1/2,1]$. Let $u \equiv f \equiv 0$, and consider
\begin{align}
    u^\nu(x,t) &:= \frac{x^\perp}{|x|^2} \sin\left(\frac{|x|}{\nu^{1/3}}\right)\phi(|x|)\gamma(t), \label{eq:cex1}\\
    f^\nu &:= \partial_t u^\nu - \nu \Delta u^\nu. \label{eq:cex2}
\end{align}
Of course, $u$ is an energy balanced solution of the incompressible Euler equations, with forcing $f$ (both vanishing identically). We next verify that $u^\nu$ is a physical realization of $u$: It is straightforward to show that $u^\nu \wstarto u \equiv 0$ in $L^\infty_t L^2_x$. Furthermore, we have
$u^\nu(\slot,0) \equiv 0$, since $\gamma(0) = 0$ by assumption. In particular, this implies that $u^\nu_0 \to u_0$ in $L^2_x$. Next, by construction of $f^\nu = \partial_t u^\nu - \nu \Delta u^\nu$, one readily verifies that
\[
f^\nu = \frac{x^\perp}{|x|^2} \sin\left( \frac{|x|}{\nu^{1/3}} \right) \phi(|x|) \gamma'(t)
+ O(\nu^{1/3}),
\]
where the $O(\nu^{1/3})$-term is with respect to $L^\infty_tL^\infty_x$. Thus, $f^\nu \weaklyto 0$ in $L^2_tL^2_x$ as $\nu \to 0$ by the oscillatory nature of the sine-factor. Finally, we point out that $u^\nu$ is a solution of the Navier-Stokes equations \eqref{eq:ns} with forcing $f^\nu$: Indeed, any velocity field of the form $U(x) = \frac{x^\perp}{|x|^2} \sin\left(\frac{|x|}{\nu^{1/3}}\right)\phi(|x|)$ is a stationary solution of the Euler equations (see e.g. \cite[Chapter 2.2.1, Example 2.1]{MajdaBertozzi}), i.e. there exists a pressure $P$ such that $U \cdot \nabla U + \nabla P = 0$. Thus, $u^\nu(x,t) = U(x) \gamma(t)$ solves the Navier--Stokes equations
\[
    \partial_t u^\nu + u^\nu \cdot \nabla u^\nu + \nabla p^\nu = \nu \Delta u^\nu + f^\nu,
\]
where $p^\nu = P(x)\gamma(t)^2$, and $f^\nu = \partial_t u^\nu - \nu \Delta u^\nu$. Lastly, we remark that, even though $u^\nu \weaklyto u\equiv 0$, it is readily verified from \eqref{eq:cex1}, that $\Vert u^\nu \Vert_{L^2_x}^2 \not \to 0$ as $\nu \to 0$, and hence $u^\nu $ does not converge strongly to $0 \equiv u$ in $L^2_tL^2_x$. This establishes the claim.

\end{example}

\section{Proof of Theorem \ref{thm:ebal}}
\label{sec:ebal}

In the present section, we will provide a detailed proof of Theorem \ref{thm:ebal}. After recalling several useful {\it a priori} estimates on the Navier-Stokes equations \eqref{eq:ns} in Section \ref{sec:apriori}, a proof of the direction ``energy balance $\Rightarrow$ strong convergence $u^\nu \to u$'' will be given in Proposition \ref{prop:ebal-->conv} in Section \ref{sec:ebal-->conv}. A proof of the converse is given in Section \ref{sec:conv-->ebal}, see Proposition \ref{prop:conv-->ebal}. Finally, in Section \ref{sec:contt}, Proposition \ref{prop:contt}, we show that, under the assumptions of Theorem \ref{thm:ebal}, the convergence $u^\nu \to u$ in $L^2_tL^2_x$ can be improved to uniform-in-time convergence, i.e. $u^\nu \to u$ in $C([0,T];L^2_x)$.

\subsection{{\it A priori} estimates}
\label{sec:apriori}
We collect several useful {\it a priori} estimates for solutions of the Navier-Stokes equations. Although these estimates are well-known, we include precise statements and proofs for completeness.

\begin{lemma} \label{lem:upperbound}
Let $u^\nu \rightharpoonup u$ be a physically realizable solution of the forced Euler equations with initial data $u_0^\nu\to u_0$ in $L^2$. Assume that $\sup_{\nu} \Vert f^\nu \Vert_{L^2_tL^2_x} \le M$. Then
\[
\Vert u(t) \Vert_{L^2_x}^2
\le
\left(
\Vert u_0 \Vert_{L^2_x}^2
+
\sqrt{t} \,M
\right)
e^{\sqrt{t} \, M }, \quad \forall \, t\in [0,T].
\]
\end{lemma}

\begin{proof}
For $\nu > 0$, we have
\begin{align*}
\Vert u^\nu(t) \Vert_{L^2_x}^2
&=
\Vert u_0^\nu \Vert_{L^2_x}^2
- 2\nu \int_0^t \Vert \omega^\nu \Vert^2_{L^2_x} \, d\tau
+ 2\int_0^t \langle f^\nu, u^\nu \rangle \, d\tau
\\
&\le
\Vert u_0^\nu \Vert_{L^2_x}^2
+ 2\int_0^t \Vert f^\nu (\tau)\Vert_{L^2_x} \Vert u^\nu (\tau)\Vert_{L^2_x}  \, d\tau
\end{align*}
We estimate the forcing term as follows,
\begin{align*}
2\int_0^t \Vert f^\nu (\tau)\Vert_{L^2_x} \Vert u^\nu (\tau)\Vert_{L^2_x}  \, d\tau
&\le
\int_0^t \Vert f^\nu (\tau) \Vert_{L^2_x} \, d\tau
+
\int_0^t \Vert f^\nu (\tau) \Vert_{L^2_x}  \Vert u^\nu(\tau) \Vert_{L^2_x}^2 \, d\tau,
\end{align*}
to obtain
\begin{align*}
\Vert u^\nu(t) \Vert_{L^2_x}^2
&\le
\Vert u_0^\nu \Vert_{L^2_x}^2 + \int_0^t \Vert f^\nu (\tau) \Vert_{L^2_x} \, d\tau
%\\
%&\qquad
+ \int_0^t \Vert f^\nu (\tau) \Vert_{L^2_x}  \Vert u^\nu(\tau) \Vert_{L^2_x}^2 \, d\tau.
\end{align*}
The integral form of Gronwall's inequality then implies that
\begin{equation} \label{Gronwall}
\Vert u^\nu(t) \Vert_{L^2_x}^2
\le
\left(
\Vert u_0^\nu \Vert_{L^2_x}^2
+
\int_0^t \Vert f^\nu (\tau) \Vert_{L^2_x} \, d\tau
\right)
e^{\int_0^t \Vert f^\nu \Vert_{L^2_x} \, d\tau}.
\end{equation}
This provides a quantitative upper bound on $\Vert u^\nu (t) \Vert_{L^2_x}$, provided that $f^\nu \in L^1_t L^2_x$. The additional assumption $\sup_\nu \Vert f^\nu \Vert_{L^2_tL^2_x} \le M$ implies an estimate which is uniform in $\nu$:
\begin{align}
\label{eq:gron}
\begin{aligned}
\Vert u^\nu(t) \Vert_{L^2_x}^2
&\le
\left(
\Vert u_0^\nu \Vert_{L^2_x}^2
+
\sqrt{t} \,\Vert f^\nu \Vert_{L^2_t L^2_x}
\right)
e^{\sqrt{t} \, \Vert f^\nu \Vert_{L^2_t L^2_x} }
\\
&\le
\left(
\Vert u_0^\nu \Vert_{L^2_x}^2
+
\sqrt{t} \, M
\right)
e^{\sqrt{t} \, M }.
\end{aligned}
\end{align}
Since $u^\nu \rightharpoonup u$ weakly-$\ast$ in $L^\infty_t L^2_x$, and $u_0^\nu \to u_0$ converges strongly in $L^2$, the above estimate implies that
\[
\Vert u(t) \Vert_{L^2_x}^2
\le
\left(
\Vert u_0 \Vert_{L^2_x}^2
+
\sqrt{t} \, M
\right)
e^{\sqrt{t} \, M }.
\]
\end{proof}

\begin{remark}
If $f^\nu \to f$ strongly in $L^1_tL^2_x$, and if the convergence $u_0^\nu \to u_0$ is strong in $L^2_x$, then, using weak lower-semicontinuity of the $L^2$-norm on the estimate \eqref{Gronwall} derived in the proof of Lemma \ref{lem:upperbound}, we obtain
\[
\Vert u(t) \Vert^2_{L^2_x}
\le
\left(
\Vert u_0 \Vert^2_{L^2_x}
+
\int_0^t \Vert f(\tau) \Vert_{L^2_x} \, d\tau
\right)
e^{\int_0^t \Vert f(\tau) \Vert_{L^2_x} \, d\tau}.
\]
\end{remark}

%We note the following simple corollary of Lemma \ref{lem:upperbound}:

%\begin{corollary}
%\label{cor:upperbound}
%Let $u^\nu \to u$ be a  physically realizable solution of the forced Euler equations with initial data %$u^\nu_0 \to u_0$ in $L^2$. Assume that $\sup_{\nu} \Vert u^\nu_0 \Vert_{L^2_x}, \sup_{\nu} \Vert f^\nu %\Vert_{L^2_tL^2_x} \le M$. Then there exists a constant $C = C(T,M) > 0$, such that
%\[
%\left|
%\Vert u^\nu(t) \Vert^2_{L^2_x} - \Vert u^\nu_0 \Vert^2_{L^2_x}
%\right|
%\le C \sqrt{t}, \quad \forall \, t\in [0,T].
%\]
%\end{corollary}
%\begin{proof}
%An almost verbatim repetition of the proof of Lemma \ref{lem:upperbound} shows that, in fact,
%\begin{align*}
%\left|
%\Vert u^\nu(t) \Vert_{L^2_x}^2 - \Vert u_0^\nu \Vert_{L^2_x}^2
%\right|
%&\le
%\int_0^t \Vert f^\nu (\tau) \Vert_{L^2_x} \, d\tau
%+ \int_0^t \Vert f^\nu (\tau) \Vert_{L^2_x}  \Vert u^\nu(\tau) \Vert_{L^2_x}^2 \, d\tau.
%\end{align*}
%By \eqref{eq:gron} in the proof of Lemma \ref{lem:upperbound}, and the assumption $\Vert u^\nu_0 %\Vert_{L^2_x}\le M$, we have a uniform bound of the form $\Vert u^\nu(\tau) \Vert^2_{L^2_x} \le C_0(T, M)$. %%\begin{align}
%\left|
%\Vert u^\nu(t) \Vert_{L^2_x}^2 - \Vert u_0^\nu \Vert_{L^2_x}^2
%\right|
%&\le (1+C_0) \int_0^t \Vert f^\nu (\tau) \Vert_{L^2_x} \, d\tau \nonumber
%\\
%&\le (1+C_0) M \sqrt{t}. \label{sqrttest}
%\end{align}
%The constant $C = (1+C_0)M$ on the right-hand side depends only on $\Vert u_0 \Vert^2_{L^2_x}$, $T$, $M$.
%\end{proof}

We next show that, under the $L^2_tL^2_x$-bound on the forcing, physically realizable solutions belong to $C([0,T];w\text{-}L^2(\T^2))$.

\begin{lemma}
\label{lem:tcont}
Let $u\in L^\infty_t L^2_x$ be a physically realizable weak solution of the incompressible Euler equations. Consider a physical realization $u^\nu \weaklyto u$ with external forcing $f^\nu \weaklyto f$, such that $\sup_{\nu} \Vert f^\nu \Vert_{L^2_tL^2_x} < \infty$. Let $\langle \cdot, \cdot \rangle_{L^2}$ denote  the $L^2$-inner product. Then, up to redefinition on a set of times of Lebesgue measure zero, it follows that, for every $\phi \in L^2(\T^2)$, the function $t \mapsto \langle u(t), \phi \rangle_{L^2}$ is continuous, i.e. $u \in C([0,T];w\text{-}L^2(\T^2))$.
\end{lemma}

\begin{remark}
  Given the result of Lemma \ref{lem:tcont}, we will assume that any physically realizable solution $u^\nu \weaklyto u$ with $\sup_{\nu} \Vert f^\nu \Vert_{L^2_tL^2_x}<\infty$ belongs to $u\in C([0,T];w\text{-}L^2(\T^2))$ without further comment.
\end{remark}

\begin{proof}[Proof of Lemma \ref{lem:tcont}]

Let $M:= \sup_{\nu} \Vert f^\nu \Vert_{L^2_t L^2_x}$. It follows from the strong convergence $u^\nu_0 \to u_0$ and by Lemma \ref{lem:upperbound} that we can bound
\[
\Vert u^\nu \Vert_{L^2_x}^2 \le \left(\Vert u^\nu_0 \Vert_{L^2_x} + \sqrt{T} M \right) e^{\sqrt{T} M} \le C,
\]
by a constant $C$ that is independent of $\nu$. Therefore $\{u^\nu \}$ is a bounded subset of $L^\infty((0,T);L^2(\T^2))$.

Next we will bound $\partial_t u^\nu$.

Let $\phi = \phi(x)$ be a smooth test vector field. We write
\[
\partial_t u^\nu = -u^\nu \cdot \nabla u^\nu - \nabla p^\nu + \nu \Delta u^\nu + f^\nu,
\]
we take the inner product with $\phi$ and integrate on $\T^2$ to find
\begin{align*}
\int_{\T^2} \partial_t u^\nu \cdot \phi
&=
-\int_{\T^2} \left[ \div \left( u^\nu \otimes u^\nu \right) + \nabla p^\nu + \nu \Delta u^\nu + f^\nu \right] \cdot \phi.
\end{align*}
Without loss of generality we may assume that $\div (\phi) = 0$, since we have $\div(u^\nu)=0$; thus the pressure term drops out. Transferring derivatives to $\phi$, and bounding the terms on the right, we find
\begin{align*}
\left|\int_{\T^2} \partial_t u^\nu \cdot \phi \right|
&\le
\Vert u^\nu \Vert_{L^2_x}^2 \Vert D \phi \Vert_{L^\infty_x}
+ \nu \Vert u^\nu \Vert_{L^2_x} \Vert \Delta \phi \Vert_{L^2_x} + \Vert f^\nu \Vert_{L^2_x} \Vert \phi \Vert_{L^2_x},
\end{align*}
where $D\phi$ denotes the Jacobian matrix of $\phi$.

 It follows from Sobolev embedding that, for sufficiently large $L$, we have
\[
\left|\int_{\T^2} \partial_t u^\nu \cdot \phi \right|
\le
C\left( 1 + \Vert f^\nu \Vert_{L^2_x} \right) \Vert \phi \Vert_{H^{L}_x},
\]
for all $\phi \in H^L(\T^2)$, with a constant $C>0$ that is independent of $\nu$. Interpreting the left-hand-side above as a duality pairing between $H^{-L}$ and $H^L$ yields
\[\Vert \partial_t u^\nu \Vert_{H^{-L}_x} \leq 1 + \Vert f^\nu \Vert_{L^2_x} .\]
Therefore
\begin{equation} \label{visctderest}
\Vert \partial_t u^\nu \Vert_{L^2_tH^{-L}_x} \leq C(1+M),
\end{equation}
so that $\{\partial_t u^\nu \}$ is a bounded subset of $L^2((0,T);H^{-L}(\T^2))$.

It now follows immediately from the Aubin-Lions-Simon lemma, see for instance \cite[Theorem II.5.16]{BoyerFabrie2013}, that $\{u^\nu\}_\nu$ is a compact subset of $C([0,T]; w\text{-}L^2(\T^2))$, because the embedding $L^2 \subset w\text{-}L^2$ is compact. Given that $u^\nu \rightharpoonup u$ weak-$\ast$ $L^\infty_tL^2_x$ it follows from the uniqueness of limits that $ u \in C([0,T]; w\text{-}L^2(\T^2))$, as desired.

\end{proof}

%%%%%%%%%%%%%
\begin{lemma}
\label{lem:apriori}
Let $u^\nu \rightharpoonup u$ be a physically realizable solution of the forced Euler equations with initial data $u_0^\nu\to u_0$ strongly in $L^2_x$. Assume that $\sup_{\nu} \Vert f^\nu \Vert_{L^2_tL^2_x} \le M$. Then $u=u(t)$ is right-continuous at $t=0$, i.e.
\[
\lim_{\substack{t\to 0^+}} \Vert u(t) - u_0 \Vert_{L^2_x} = 0.
\]
\end{lemma}

\begin{proof}
Since $u\in C([0,T];w\text{-}L^2_x)$, and by lower-semicontinuity of the $L^2_x$-norm under weak limits and the upper bound from Lemma \ref{lem:upperbound}, it follows that:
\begin{align*}
\Vert u_0 \Vert_{L^2_x}^2
&\le
\liminf_{t\to 0} \Vert u(t) \Vert_{L^2_x}^2
\le
\limsup_{t\to 0} \Vert u(t) \Vert_{L^2_x}^2
\\
&\le
\limsup_{t\to 0} \left(
\Vert u_0 \Vert_{L^2_x}^2
+
\sqrt{t} \, M
\right)
e^{\sqrt{t} \, M }
=
\Vert u_0 \Vert_{L^2_x}^2.
\end{align*}
These lower and upper bounds imply that $\lim_{t\to 0} \Vert u(t) \Vert_{L^2_x} = \Vert u_0 \Vert_{L^2_x}$.
We also have $u(t) \rightharpoonup u_0$ as $t\to 0$, from $u \in C([0,T];w\text{-}L^2_x)$.
Weak convergence together with convergence of the norms implies strong convergence, which concludes the proof.
\end{proof}

\begin{remark} \label{eq:ucont}
It follows from Lemma \ref{lem:apriori}, in particular, that
\begin{equation} \label{eq:inencont}
\lim_{\delta \to 0}
\Vert u(\delta) \Vert_{L^2_x}
=
\Vert u_0 \Vert_{L^2_x}.
\end{equation}
Furthermore, since $u \in L^\infty_t L^2_x$ and $f \in L^2_t L^2_x$, we have $\langle f, u \rangle_{L^2_x} \in L^1_t$. Therefore,
\begin{equation} \label{eq:fucont}
\lim_{\delta \to 0} \int_\delta^t \langle f, u \rangle_{L^2_x} \, ds =
\int_0^t \langle f, u \rangle_{L^2_x} \, ds.
\end{equation}
\end{remark}

Following the classical terminology of turbulence theory, we refer to the square of the $L^2$-norm of vorticity as the enstrophy.

\begin{lemma}
\label{lem:vort-bd}
Let $u^\nu$ be a solution of the forced Navier-Stokes equations with forcing $f^\nu \in L^2_tL^2_x$. Then there exists a constant $C = C(\Vert u^\nu_0 \Vert_{L^2_x},\Vert f^\nu \Vert_{L^2_tL^2_x})>0$, such that
  \[
  \Vert \omega^\nu(t) \Vert_{L^2_x} \le \frac{C}{\sqrt{\nu t}}.
  \]
\end{lemma}

\begin{proof}
  From the energy balance for solutions of Navier-Stokes, we obtain
  \begin{align*}
\nu \int_0^T \Vert \omega^\nu \Vert^2_{L^2_x} \, dt
&=
\frac12 \Vert u_0 \Vert^2_{L^2_x} - \frac12 \Vert u^\nu(T) \Vert_{L^2_x}^2 - \int_0^T \langle f^\nu, u^\nu \rangle_{L^2_x} \, dt
\\
&\le
\frac12 \Vert u_0 \Vert^2_{L^2_x} + \frac12 \Vert f^\nu \Vert_{L^2_tL^2_x}^2 + \frac{T}{2}  \Vert u^\nu \Vert_{L^\infty_t L^2_x}^2.
\end{align*}
  We use the {\it a priori} estimate from Lemma \ref{lem:upperbound} to deduce that all terms on the right-hand side are bounded by a positive constant $C = C(\Vert u^\nu_0 \Vert_{L^2_x},\Vert f^\nu \Vert_{L^2_tL^2_x})$:
  \begin{align} \label{eq:omC}
    \nu \int_0^T \Vert \omega^\nu \Vert^2_{L^2_x} \, d\tau
    \le
    C(\Vert u^\nu_0 \Vert_{L^2_x},\Vert f^\nu \Vert_{L^2_tL^2_x}).
  \end{align}
  To obtain a pointwise estimate on $\Vert \omega^\nu(t) \Vert_{L^2_x}$, we note that the vorticity equation implies the following upper bound on the enstrophy (cf. \eqref{eq:enst}), for $\tau \in [0,t]$:
  \begin{align*}
    \Vert \omega^\nu(t) \Vert_{L^2_x}^2
    &\le
    \Vert \omega^\nu(\tau) \Vert_{L^2_x}^2
    + \frac{1}{\nu} \int_\tau^t \Vert f^\nu(s) \Vert_{L^2_x}^2 \, ds.
    \\
      &\le
    \Vert \omega^\nu(\tau) \Vert_{L^2_x}^2
    + \frac{1}{\nu} \Vert f^\nu \Vert_{L^2_t L^2_x}^2.
\end{align*}
  In particular, integrating in $\tau$ from $0$ to $t$, the above estimate implies that
  \begin{align*}
    \nu t \, \Vert \omega^\nu(t) \Vert_{L^2_x}^2
    &\le
    \nu \int_0^t \Vert \omega^\nu \Vert_{L^2_x}^2 \, d\tau + T \Vert f^\nu \Vert_{L^2_tL^2_x}^2
  \end{align*}
  By \eqref{eq:omC}, the first term on the right-hand side is bounded by a constant depending only on the initial data and forcing. In particular, we conclude that there exists a positive constant $C = C(T,\Vert u^\nu_0 \Vert_{L^2_x},\Vert f^\nu \Vert_{L^2_tL^2_x})$, such that
  \[
  \Vert \omega^\nu(t) \Vert_{L^2_x} \le \frac{C}{\sqrt{\nu t}}.
  \]
\end{proof}

\subsection{Energy balance implies strong convergence}
\label{sec:ebal-->conv}

\begin{proposition}
  \label{prop:ebal-->conv}
  %\rem{[to be filled]}
  Let $u$ be a physically realizable solution of the forced Euler equations \eqref{eq:euler}, with initial data $u_0 \in L^2(\T^2)$ and forcing $f\in L^2 ((0,T);L^2 (\T^2))$. Let $u^\nu$ be a physical realization of $u$, and assume additionally that the forcing $f^\nu$ converges strongly to $f$ in $L^2 ((0,T);L^2 (\T^2))$. Then it holds that, if $u$ is energy balanced, then $u^\nu \to u$ strongly in $L^2 ((0,T);L^2 (\T^2))$.
\end{proposition}

\begin{proof}
  As $u$ is energy balanced it follows that
  \[
  \Vert u(t) \Vert_{L^2_x}^2 =   \Vert u_0 \Vert_{L^2_x}^2 + 2 \int_0^t \langle f, u \rangle_{L^2_x} \, d\tau.
  \]
  Furthermore, in view of the convergence $u^\nu \weaklyto u$  in weak-$\ast$ $L^\infty((0,T);L^2(\T^2))$ we have
  \[u^\nu \weaklyto u \text{ weak} \, L^2_tL^2_x,\]
  so that, by weak lower semicontinuity,
  \[
  \Vert u \Vert_{L^2_tL^2_x} \le \liminf_{\nu \to 0^+} \Vert u^\nu \Vert_{L^2_tL^2_x}.
  \]

  Recall the energy inequality for the physical realization, valid for any $0 \le \tau \le T$:
\[
  \Vert u^\nu(\tau) \Vert_{L^2_x}^2 \leq   \Vert u_0^\nu \Vert_{L^2_x}^2 + 2 \int_0^\tau \langle f^\nu, u^\nu \rangle_{L^2_x} \, ds.
  \]
 The information above yields
\begin{align}\label{enbalstrngconv}
  \int_0^T \Vert u(\tau) \Vert_{L^2_x}^2 \,d\tau  \le & \liminf_{\nu \to 0^+} \int_0^T \Vert u^\nu (\tau) \Vert_{L^2_x}^2 \,d\tau \\
  \le & \limsup_{\nu \to 0^+} \int_0^T \Vert u^\nu (\tau) \Vert_{L^2_x}^2 \,d\tau  \\
  \le & \limsup_{\nu \to 0^+} \int_0^T \left(\Vert u_0^\nu \Vert_{L^2_x}^2 + 2 \int_0^\tau \langle f^\nu, u^\nu \rangle_{L^2_x} \, ds \right) \, d\tau  \\
   = & \int_0^T \left( \Vert u_0 \Vert_{L^2_x}^2 + 2 \int_0^\tau \langle f , u \rangle_{L^2_x} \, ds \right) \, d\tau\\
   = & \int_0^T \Vert u(\tau) \Vert_{L^2_x}^2 \,d\tau .
\end{align}
Therefore $\Vert u^\nu \Vert_{L^2_tL^2_x} \to \Vert u  \Vert_{L^2_tL^2_x}$ and, using again that convergence of norms and weak convergence implies strong convergence, the proof is concluded.
\end{proof}

\begin{remark} \label{rem:anydimension}
Notice that Proposition \ref{prop:ebal-->conv} is valid in any space dimension. In other words, after appropriately adjusting the definition of an energy balanced physically realizable solution, we may substitute $\T^2$ in the statement by $\T^d$ for any $d \geq 2$.
\end{remark}

\subsection{Strong convergence implies energy balance}
\label{sec:conv-->ebal}

The proof that strong convergence $u^\nu \to u$ in $L^2_tL^2_x$ of the physical realization implies energy balance of the limit $u$ will be based on the following inequality for the enstrophy, for $\delta < t$, with $\delta,t\in [0,T]$, which follows from \eqref{eq:enst}:
\[
\Vert \omega^\nu(t) \Vert_{L^2_x}^2
\le
\Vert \omega^\nu(\delta) \Vert_{L^2_x}^2
-
\nu \int_\delta^t \Vert \nabla \omega^\nu \Vert^2 \, d\tau
+
\frac{1}{\nu} \int_\delta^t \Vert f^\nu \Vert_{L^2_x}^2 \, d\tau.
\]
The basic idea, introduced in \cite{CLNS2016}, is to reduce the inequality above to a differential inequality for $\Vert \omega^\nu \Vert_{L^2_x}^2$, which in turn can be used to show that the energy dissipation term $\nu \int_0^T \Vert \omega^\nu \Vert_{L^2_x}^2 \, d\tau \to 0$ as $\nu \to 0$. To this end, the crucial ingredient is a good lower bound on $\Vert \nabla \omega^\nu \Vert^2_{L^2_x}$ in terms of $\Vert \omega^\nu \Vert_{L^2_x}^2$, which is obtained in Lemma \ref{lem:interpol} below. Before we state this lemma we must recall the notation for structure functions introduced in \cite{LMPP2021a}.

If $v \in L^2_x$ then the ($L^2$-based) structure function $S_2(v;r)$ is given by
\[S_2(v;r)= \left( \fint_{B_r(0)} |v(x+h)-v(x)|^2 \, dh \right)^{1/2}.\]
If, now, $v \in L^2_tL^2_x$ then the time-integrated structure function $S^T_2(v;r)$ is defined as
\[S^T_2(v;r) = \left(\int_0^T [S_2(v(t);r)]^2 \, dt\right)^{1/2}.\]

\begin{lemma}
\label{lem:interpol}
Let $\{u^\nu\}_{\nu > 0}$ be a precompact family of divergence-free vector fields in $L^2 ((0,T);L^2 (\T^2))$. There exists a monotonically increasing function $\sigma: [0,\infty) \to [0, \infty)$, such that $\lim_{z\to \infty} \sigma(z) = \infty$, and such that for each $\nu >0$ and $\delta,t\in [0,T]$, $\delta < t$, the vorticity $\omega^\nu = \curl  u^\nu $ satisfies the following inequality
\[
\left(
\int_\delta^t \Vert \omega^\nu(\tau) \Vert_{L^2_x}^2 \, d\tau
\right)^2
\sigma\left(
\int_\delta^t \Vert \omega^\nu(\tau) \Vert_{L^2_x}^2 \, d\tau
\right)
\le
\int_\delta^t \Vert \nabla \omega^\nu(\tau) \Vert_{L^2_x}^2 \, d\tau.
\]
\end{lemma}

\begin{proof}
  This result is implicitly contained in the proof of \cite[Thm. 2.11]{LMPP2021a}. We outline the argument here. The first step is the following ``interpolation-type'' inequality valid for any $\omega = \curl u \in H^1 (\T^2)$, see \cite[Lemma 2.6]{LMPP2021a}: There exists an absolute constant $C>0$, such that
  \[
  \Vert \omega \Vert_{L^2_x} \le C r \Vert \nabla \omega \Vert_{L^2_x} + \frac{2S_2(u;r)}{r}, \quad \forall r>0.
  \]
  Applying this estimate to $\omega^\nu$, squaring terms and integrating in time, it follows that
  \[
  \int_\delta^t \Vert \omega^\nu \Vert_{L^2_x}^2 \, d\tau
  \lesssim
  r^2 \int_\delta^t \Vert \nabla \omega^\nu \Vert_{L^2_x}^2 \, d\tau
  +
  \frac{[S_2^T(u^\nu;r)]^2}{r^2}, \quad \forall \, r>0,
  \]
  with an implied constant independent of $\nu,r>0$ and independent of $\delta,t \in [0,T]$. By \cite[Prop. 2.10]{LMPP2021a}, the precompactness of $\{u^\nu\}\subset L^2_tL^2_x$ implies that there exists a (monotonically increasing) modulus of continuity $\phi: [0,\infty) \to [0,\infty)$, with $\lim_{r\to 0} \phi(r) = 0$, such that
      \[
      \sup_{\nu > 0} [S^T_2(u^\nu;r)]^2 \le [\phi(r)]^2.
      \]
      Since the left-hand side is uniformly bounded from above by
      \[\sup_\nu S^T_2(u^\nu;r) \le \sup_\nu 2\Vert u^\nu \Vert_{L^2_tL^2_x} <\infty,\]
      we may assume, without loss of generality, that $[\phi(r)]^2 \le \beta$ for some $\beta >0$, for all $r \ge 0$. Optimizing with respect to $r > 0$  so as to balance terms (cf. \cite[eq. (2.13)]{LMPP2021a})  results in
  \[
  \left(\int_\delta^t \Vert \omega \Vert_{L^2_x}^2 \, d\tau\right)^2
  \le C \left[
  \phi\left(\beta\left[\textstyle\int_\delta^t \Vert \nabla \omega^\nu \Vert_{L^2_x}^2 \, d\tau \right]^{-1/4}\right)\right]^2 \int_\delta^t \Vert \nabla \omega^\nu \Vert_{L^2_x}^2 \, d\tau,
  \]
  with a constant $C>0$, independent of $\nu$, $r$, $\delta$ and $t$.
 To simplify notation, let us define a new modulus of continuity $\tilde{\phi} := C \phi^2$, so that
  \begin{equation}\label{tildephi}
  \left(\int_\delta^t \Vert \omega \Vert_{L^2_x}^2 \, d\tau\right)^2
  \le
  \tilde{\phi}\left(\beta\left[\textstyle\int_\delta^t \Vert \nabla \omega^\nu \Vert_{L^2_x}^2 \, d\tau \right]^{-1/4}\right) \int_\delta^t \Vert \nabla \omega^\nu \Vert_{L^2_x}^2 \, d\tau.
    \end{equation}

  We next claim that the right-hand side of inequality \eqref{tildephi} is bounded by  a monotonically increasing function of $z = \int_\delta^t \Vert \nabla \omega^\nu \Vert^2_{L^2_x} \, d\tau$. Indeed, if we denote the right-hand side of \eqref{tildephi} by $f(z) = \tilde{\phi}(\beta z^{-1/4}) z$, then $f(z) = o(z)$ as $z \to \infty$, since $\tilde{\phi}$ is a modulus of continuity and hence $\tilde{\phi}(\beta z^{-1/4}) \to 0$ as $z\to \infty$. In \cite[Appendix C, Lemma C.1]{LMPP2021a} it is shown that, for such $f=f(z)$, there exists a dominating function $F$, with $F(z) \ge f(z)$, satisfying the following properties: $F(z) = o(z)$ as $z \to \infty$, $F$ is invertible and its inverse $F^{-1}$ is a monotonically increasing function which grows super-linearly. It is shown, furthermore, that $F^{-1}$ can be written in the form $F^{-1}(y) = y\sigma(\sqrt{y}) $, with $\sigma$ a monotonically increasing function such that $\sigma(\sqrt{y}) \to \infty$ as $y\to \infty$. Using the notation introduced in the current paragraph, see \eqref{tildephi}, and estimating $f$ by $F$ we find
      \[
  \left(\int_\delta^t \Vert \omega \Vert_{L^2_x}^2 \, d\tau\right)^2
    \le
  F\left( \int_\delta^t \Vert \nabla \omega^\nu \Vert_{L^2_x}^2 \, d\tau \right).
  \]
Applying $F^{-1}$ to both sides of the inequality above yields
  \[
    F^{-1}\left(\left[\int_\delta^t \Vert \omega \Vert_{L^2_x}^2 \, d\tau\right]^2
    \right)
  \le
  \int_\delta^t \Vert \nabla \omega^\nu \Vert_{L^2_x}^2 \, d\tau.
  \]
  Finally, writing $F^{-1}(y) = y\sigma(\sqrt{y})$  implies the desired upper bound
  \[
  \left(\int_\delta^t \Vert \omega \Vert_{L^2_x}^2 \, d\tau\right)^2
  \sigma\left(\int_\delta^t \Vert \omega \Vert_{L^2_x}^2 \, d\tau\right)
  \le
  \int_\delta^t \Vert \nabla \omega^\nu \Vert_{L^2_x}^2 \, d\tau.
  \]

\end{proof}

The previous lemma will allow us to derive a differential inequality for the energy dissipation term $\zeta_\nu(t) = \nu \int_0^t \Vert \omega^\nu \Vert_{L^2_x}^2 \, d\tau$, starting from the vorticity equation. The next lemma will then be used to show that $\zeta_\nu(T) \to 0$, as $\nu \to 0$, i.e. absence of anomalous dissipation, see Definition \ref{def:anomdissinvdiss}.

\begin{lemma} \label{lem:ide}
Let $0\leq a < T$, $M>0$ and let $\sigma: \R_+ \to \R_+$ be a continuous, monotonically increasing function, such that $\lim_{z\to \infty} \sigma(z) = \infty$. Let $\{\zeta_\nu\}_{\nu>0} \subset W^{1,1}([a,T])$ be a family of monotonically increasing functions. If $\zeta_\nu$ satisfies the differential inequality
\begin{align*}
\frac{d\zeta_\nu}{dt}
\le M - \zeta_\nu^2 \sigma\left(\frac{\zeta_\nu}{\nu}\right), \quad \text{ a.e. } t \in [a,T],
\end{align*}
then $\limsup_{\nu \to 0} \zeta_\nu(T) = 0$.
\end{lemma}

\begin{proof}
Since $\zeta_\nu$ is monotonically increasing and $\zeta_\nu \in W^{1,1}$ we have, for almost every $t\in [a,T]$, that:
\[
0 \le \frac{d}{dt} \zeta_\nu(t) \le M - [\zeta_\nu(t)]^2 \sigma\left(\frac{\zeta_\nu (t)}{\nu}\right).
\]
We note that the function on the right-hand side is continuous as a function of $t$. Letting $t\to T$, we deduce that
\[
[\zeta_\nu(T)]^2 \sigma(\zeta_\nu(T)/\nu) \le M, \quad \forall \, \nu > 0.
\]
From the monotonicity of $\sigma$ it follows that the function $z \mapsto z^2 \sigma(z/\nu)$  is increasing.

Let us assume now, by contradiction, that $\limsup_{\nu \to 0} \zeta_\nu(T) \ge 2\epsilon_0 > 0$, so that there exists a sequence $\nu_k \to 0$ with $\zeta_{\nu_k}(T) \ge \epsilon_0$, for all $k\in \N$. In this case we have, in particular, that
\begin{align}
\label{eq:upbd}
\epsilon_0^2 \sigma\left(\frac{\epsilon_0}{\nu_k}\right)
\le
[\zeta_{\nu_k}(T)]^2 \sigma\left(\frac{\zeta_{\nu_k}(T)}{\nu_k}\right)
\le M,
\end{align}
for all $k\in \N$, i.e. $\epsilon_0^2 \sigma(\epsilon_0/\nu_k)$ is uniformly bounded. On the other hand, letting $k \to \infty$, and using the assumption that $\lim_{z\to \infty} \sigma(z) = \infty$ leads to
\[
\lim_{k\to \infty} \epsilon_0^2 \sigma\left(\frac{\epsilon_0}{\nu_k}\right)
= \infty,
\]
in contradiction with \eqref{eq:upbd}.
 Therefore, we must have $\limsup_{\nu \to 0} \zeta_{\nu}(T) = 0$ as claimed.
\end{proof}

We are now in a position to prove the following result, which encodes the ``strong convergence $\Rightarrow$ energy balance'' part of Theorem \ref{thm:ebal}.

\begin{proposition}
    \label{prop:conv-->ebal}
Let $u$ be a physically realizable solution of the Euler equations \eqref{eq:euler}  with forcing $f\in L^2 ((0,T);L^2 (\T^2))$. Let $u^\nu$ be a physical realization of $u$ and assume that the forcing $f^\nu$ converges weakly to $f$ in $L^2 ((0,T);L^2 (\T^2))$. If $u^\nu \to u$ strongly in $L^2 ((0,T);L^2 (\T^2))$, then $u$ is energy balanced.
\end{proposition}

\begin{proof}
\textbf{Step 1:}
We begin by proving that for any $\delta > 0$, the strong convergence assumption implies that the ``$\delta$-truncated'' energy dissipation term is vanishingly small:
\[
\nu \int_\delta^T \Vert \omega^\nu \Vert^2_{L^2_x}  \, d\tau
\to 0, \quad
\text{as } \nu \to 0.
\]
To see this, first recall that, by Lemma \ref{lem:vort-bd}, for any $\delta > 0$, we have $\Vert \omega^\nu (\delta) \Vert_{L^2_x} < \infty$. We may thus consider the following enstrophy equation, valid for any $\nu > 0$, and $\delta \in (0,t]$:
\[
\frac12 \Vert \omega^\nu(t) \Vert_{L^2_x}^2
=
\frac12 \Vert \omega^\nu(\delta) \Vert_{L^2_x}^2
-
\nu \int_\delta^t \Vert \nabla \omega^\nu \Vert_{L^2_x}^2 \, d\tau
+ \int_{\delta}^t \langle \omega^\nu, g^\nu \rangle \, d\tau,
\]
with  $g^\nu = \curl f^\nu \in L^2_tH^{-1}_x$. Integrating by parts the last term and using the Cauchy-Schwarz inequality we find
\[
\Vert \omega^\nu(t) \Vert_{L^2_x}^2
\le
\Vert \omega^\nu(\delta) \Vert_{L^2_x}^2
-
2\nu \int_\delta^t \Vert \nabla \omega^\nu \Vert_{L^2_x}^2 \, d\tau
+ 2\int_{\delta}^t \Vert \nabla \omega^\nu \Vert_{L^2_x} \Vert f^\nu \Vert_{L^2_x} \, d\tau.\]
By Young's inequality we have
\begin{equation} \label{vortestnew}
\Vert \omega^\nu(t) \Vert_{L^2_x}^2 \le
\Vert \omega^\nu(\delta) \Vert_{L^2_x}^2
-
\nu \int_\delta^t \Vert \nabla \omega^\nu \Vert_{L^2_x}^2 \, d\tau
+ \frac{1}{\nu} \int_{\delta}^t \Vert f^\nu \Vert_{L^2_x}^2 \, d\tau.
\end{equation}

In the following we keep $\delta > 0$ fixed, and we will only consider values $t \ge \delta$. We note that, using again Lemma \ref{lem:vort-bd}, there exists a constant $C >0$, depending only on $\Vert u_0^\nu\Vert_{L^2_x}$ and $\Vert f^\nu \Vert_{L^2_tL^2_x}$, such that
\begin{align} \label{eq:vortin}
\Vert \omega^\nu(\delta) \Vert^2_{L^2}
\le \frac{C}{\nu \delta }.
\end{align}
Furthermore we have
\begin{align} \label{eq:forc}
\frac{1}{\nu} \int_\delta^t \Vert f^\nu \Vert^2_{L^2} \, d\tau
\le
\frac{1}{\nu}\Vert f^\nu \Vert_{L^2_tL^2_x}^2 .
\end{align}
Using \eqref{eq:vortin}, \eqref{eq:forc} in \eqref{vortestnew}, and since $\Vert u_0^\nu \Vert_{L^2_x}$, $\Vert f^\nu \Vert_{L^2_tL^2_x}$ are clearly uniformly bounded, we can find a constant $M = M(\delta)> 0$  such that
\begin{align} \label{eq:vort2}
\Vert \omega^\nu(t) \Vert_{L^2_x}^2
&\le
\frac{M}{\nu}
- \nu \int_\delta^t \Vert \nabla \omega^\nu \Vert_{L^2_x}^2 \, d\tau,
\end{align}
for all $\nu > 0$.
To estimate the gradient term, we note that Lemma \ref{lem:interpol} implies the existence of a monotonically increasing continuous, nonnegative, function $\sigma=\sigma(z)$, with $\lim_{z\to \infty} \sigma(z) = \infty$, such that
\begin{align} \label{eq:diff}
\left(
\int_\delta^t \Vert \omega^\nu(s) \Vert_{L^2_x}^2 \, d\tau
\right)^2
\sigma\left(
\int_\delta^t \Vert \omega^\nu(s) \Vert_{L^2_x}^2 \, d\tau
\right)
\le
\int_\delta^t \Vert \nabla \omega^\nu(s) \Vert_{L^2_x}^2 \, d\tau,
\end{align}
for all $\nu > 0$. We introduce the shorthand notation
\begin{align} \label{eq:zetadef}
\zeta_{\nu,\delta}(t) := \nu \int_\delta^t \Vert \omega^\nu(\tau) \Vert_{L^2_x}^2 \, d\tau,
\end{align}
for $t \ge \delta$, and we conclude that
\begin{align*}
-\nu \int_\delta^t \Vert \nabla \omega^\nu(\tau) \Vert_{L^2_x}^2 \, d\tau
&\le
-\frac{1}{\nu} [\zeta_{\nu,\delta}(t)]^2
\sigma\left(\frac{\zeta_{\nu,\delta}(t)}{\nu}
\right).
\end{align*}
We also note that $\zeta_{\nu,\delta}\in W^{1,1}([\delta,T])$, with $\frac{d}{dt} \zeta_{\nu,\delta}(t) = \nu \Vert \omega^\nu(t) \Vert^2_{L^2}$. Multiplication of \eqref{eq:vort2} by $\nu$ and substitution of the above estimate therefore yields
\begin{align} \label{eq:diffineq}
\frac{d}{dt} \zeta_{\nu,\delta}(t)
&\le
M
- [\zeta_{\nu,\delta}(t)]^2
\sigma\left(\frac{\zeta_{\nu,\delta}(t)}{\nu}
\right).
\end{align}

Recall that $M$ depends on $\delta$ but is independent of $\nu$ and $t$. By construction, see \eqref{eq:zetadef}, $t\mapsto \zeta_{\nu,\delta}(t)$ is monotonically increasing, and $\zeta_{\nu,\delta} \in W^{1,1}([\delta,T])$. Since $\delta > 0$ is fixed, and $\zeta_{\nu,\delta}$ satisfies the differential inequality \eqref{eq:diffineq}, we can use Lemma \ref{lem:ide} to obtain that
\[\limsup_{\nu \to 0} \nu \int_\delta^T \Vert \omega^\nu (\tau) \Vert^2_{L^2_x}  \, d\tau \equiv \limsup_{\nu \to 0} \zeta_{\nu,\delta}(T) = 0,\]
as desired.
This  concludes step 1 of the proof.

\textbf{Step 2:} We are now in a position to show that the physically realizable solution $u$ satisfies the energy balance equation \eqref{eq:ebal}. First we note that the strong convergence $u^\nu \to u$ in $L^2_tL^2_x$ implies strong convergence of $\Vert u^\nu(t) \Vert_{L^2_x} \to \Vert u(t) \Vert_{L^2_x}$ in $L^2([0,T])$. Passing to subsequences as needed, without relabeling, we may assume that
\begin{align}
  \label{eq:ptconv}
  \Vert u^\nu(t) \Vert_{L^2_x} \to \Vert u(t) \Vert_{L^2_x}, \quad \text{a.e. } t\in [0,T].
\end{align}

Next, we recall the energy identity for $u^\nu$, see also \eqref{eq:ns-ebal}:
\begin{equation} \label{eq:ns-ebal-delta}
\frac12 \Vert u^\nu(t) \Vert^2
=
\frac12 \Vert u^\nu(\delta) \Vert^2
-
\nu \int_\delta^t \Vert \omega^\nu \Vert^2_{L^2_x} \, d\tau
+
\int_\delta^t \langle f^\nu, u^\nu \rangle_{L^2_x} \, d\tau,
\end{equation}
valid for $0\le \delta \le t \le T$

By \eqref{eq:ptconv} we can choose $\delta > 0$ outside of a set of measure $0$ so that $\Vert u^\nu(\delta) \Vert_{L^2_x} \to \Vert u(\delta) \Vert_{L^2_x}$. From the weak convergence $f^\nu \rightharpoonup f$ in $L^2_tL^2_x$, and the strong convergence $u^\nu \to u$ in $L^2_tL^2_x$, it follows that
\[
\int_\delta^t \langle f^\nu, u^\nu \rangle_{L^2_x} \, d\tau
\to
\int_\delta^t \langle f, u \rangle_{L^2_x} \, d\tau,
\quad
(\nu \to 0),
\]
for any $t\in [\delta,T]$. By Step 1, it follows that
\[
\nu \int_\delta^t \Vert \omega^\nu \Vert^2_{L^2_x} \, d\tau
\to 0, \quad (\nu \to 0),
\]
uniformly for all $t\in [\delta,T]$.

Thus, we conclude that for almost every $t\in [\delta,T]$, we have
\begin{align*}
\frac12 \Vert u(t) \Vert^2_{L^2_x}
&=
\lim_{\nu \to 0} \frac12 \Vert u^{\nu}(t) \Vert^2_{L^2_x}
=
\frac12 \Vert u(\delta) \Vert^2_{L^2_x}
+
\int_\delta^t \langle f, u \rangle_{L^2_x} \, d\tau.
\end{align*}
By Remark \ref{eq:ucont} it holds that $\delta\mapsto \Vert u(\delta) \Vert_{L^2_x}^2$, and $\delta \mapsto \int_\delta^t \langle f, u \rangle_{L^2_x} \, d\tau$ are both right-continuous at $\delta=0$, see \eqref{eq:inencont} and \eqref{eq:fucont}. Letting $\delta \to 0$, we therefore conclude that
\[
\frac12 \Vert u(t) \Vert^2_{L^2_x}
=
\frac12 \Vert u_0 \Vert^2_{L^2_x}
+
\int_0^t \langle f, u \rangle_{L^2_x} \, d\tau,
\]
for almost all $t\in [0,T]$, so that $u$ is energy balanced.

Redefining $u$ on a set of times of measure zero it follows that the energy balance holds for all $t \in [0,T]$, see also Remark \ref{rem:strongtcont}. This concludes the proof.
\end{proof}

\begin{corollary}
\label{cor:vanishing-enstrophy}
Let $u$ be a physically realizable solution of the Euler equations \eqref{eq:euler} with forcing $f \in L^2((0,T);L^2(\T^2))$. Let $u^\nu$ be a physical realization of $u$ and assume that the forcing $f^\nu$ converges weakly to $f$ in $L^2((0,T);L^2(\T^2))$. If $u^\nu \to u$ strongly in $L^2((0,T);L^2(\T^2))$, then
\begin{align}
\label{eq:vanishing-enstrophy}
\lim_{\nu \to 0} \, \nu \int_0^T \Vert \omega^\nu \Vert^2_{L^2_x} \, d\tau = 0.
\end{align}
\end{corollary}

\begin{proof}
For the sake of contradiction, assume that the claim is false. Then there exists $\epsilon_0 > 0$, and a sequence $\nu_n \to 0$, such that
\[
\nu_n \int_0^T \Vert \omega^{\nu_n} \Vert_{L^2_x}^2 \, d\tau \ge \epsilon_0 > 0.
\]
In step 1 of the proof of Proposition \ref{prop:conv-->ebal}, we have already shown that for any $\delta > 0$, we have
\[
\lim_{\nu\to 0} \, \nu \int_{\delta}^T \Vert \omega^\nu \Vert^2_{L^2_x} \, d\tau = 0.
\]
Thus, the assumed lower bound on the energy dissipation would imply that
\begin{align}
\label{eq:contrad0}
\liminf_{n\to \infty} \nu_n \int_0^\delta \Vert \omega^{\nu_n} \Vert_{L^2_x}^2 \, d\tau \ge \epsilon_0 > 0,
\end{align}
for any $\delta > 0$. Our aim is to find $\delta>0$ for which this lower bound fails, thus reaching the desired contradiction.

By energy balance for Navier-Stokes \eqref{eq:ns-ebal}, we have
\[
 \nu_n \int_0^\delta \Vert \omega^{\nu_n} \Vert^2_{L^2_x} \, d\tau
 = \frac12 \Vert u^{\nu_n}(\delta) \Vert_{L^2_x}^2 - \frac12 \Vert u^{\nu_n}_0 \Vert_{L^2_x}^2 - \int_0^\delta \langle f^{\nu_n}, u^{\nu_n} \rangle_{L^2_x} \, d\tau.
\]
Since $u^{\nu_n} \to u$ strongly in $L^2_tL^2_x$, and since $f^{\nu_n} \weaklyto f$ weakly in $L^2_tL^2_x$, it follows that
\[
\lim_{n\to \infty} \int_0^\delta \langle f^{\nu_n}, u^{\nu_n} \rangle_{L^2_x} \, d\tau
=
\int_0^\delta \langle f, u \rangle_{L^2_x} \, d\tau.
\]
Furthermore, since $u_0^{\nu_n} \to u(0)$ strongly in $L^2_x$ by assumption, we have
\[
\lim_{n\to \infty}
\frac12 \Vert u^{\nu_n}_0 \Vert_{L^2_x}^2 = \frac12 \Vert u(0) \Vert_{L^2_x}^2.
\]
Finally, as shown in the proof of Proposition \ref{prop:conv-->ebal} (cf. equation \ref{eq:ptconv}), after passing to a subsequence we can ensure that
\[
\lim_{n\to \infty}\Vert u^{\nu_n}(\delta) \Vert_{L^2_x} = \Vert u(\delta) \Vert_{L^2_x}, \quad \text{a.e. } \delta \in [0,T].
\]
Thus, for almost every $\delta > 0$, we conclude that
\[
\lim_{n\to \infty}
\nu_n \int_0^\delta \Vert \omega^{\nu_n}\Vert_{L^2_x}^2 \, d\tau
=
\frac12 \Vert u(\delta) \Vert^2_{L^2_x} - \frac12 \Vert u(0) \Vert^2_{L^2_x} - \int_0^\delta \langle f, u \rangle \, d\tau.
\]
From Remark \ref{eq:ucont}, it follows that the right-hand side in the last display tends to $0$ as $\delta \to 0$. In particular, given $\epsilon_0>0$, we can find $\delta > 0$, outside of a set of measure zero, such that
\begin{align}
\lim_{n\to \infty}
\nu_n \int_0^\delta \Vert \omega^{\nu_n}\Vert_{L^2_x}^2 \, d\tau
&=
\frac12 \Vert u(\delta) \Vert^2_{L^2_x} - \frac12 \Vert u(0) \Vert^2_{L^2_x} - \int_0^\delta \langle f, u \rangle \, d\tau \notag
\\
&\le \frac{\epsilon_0}{2}. \label{eq:contrad1}
\end{align}
This upper bound \eqref{eq:contrad1} is clearly in contradiction with \eqref{eq:contrad0}, completing the proof.
\end{proof}

\begin{remark} \label{rem:invdissANDnoanomdiss}
Corollary \ref{cor:vanishing-enstrophy} corresponds to the statement that strong convergence of the physical realization implies no anomalous dissipation. In Proposition \ref{prop:conv-->ebal} we use the strong convergence and the absence of anomalous dissipation to establish that there is no inviscid dissipation in the associated physically realizable solution. Proposition \ref{prop:ebal-->conv} corresponds to the statement that absence of inviscid dissipation implies strong convergence, which in turn implies no anomalous dissipation.

It is natural to ask whether absence of anomalous dissipation alone implies no inviscid dissipation. Going back to \eqref{eq:ns-ebal} we see that
\begin{align} \label{eq:sandwich}
\int_0^t \langle f^\nu, u^\nu \rangle_{L^2_x} \, d\tau  - \nu \int_0^t \Vert \omega^\nu(\tau) \Vert^2_{L^2_x} \, d\tau & = \frac12\left(\Vert u^\nu(t) \Vert^2_{L^2_x}
  - \Vert u_0^\nu \Vert_{L^2_x}^2\right)\\
& \leq
  \int_0^t \langle f^\nu, u^\nu \rangle_{L^2_x} \, d\tau   . \nonumber
\end{align}
If $f^\nu \to f$ strongly in $L^2_t L^2_x$ and $u^\nu \rightharpoonup u$ weakly in $L^2_t L^2_x$ then the right-hand-side of \eqref{eq:sandwich} converges to
\[\int_0^t \langle f , u  \rangle_{L^2_x} \, d\tau ,\]
as $\nu \to 0$. If, additionally, we assume absence of anomalous dissipation then the left-hand-side converges to the same term. It follows that
\begin{equation} \label{eq:invdissANDnoanomdiss}
\limsup_{\nu \to 0} \left( \Vert u^\nu(t) \Vert^2_{L^2_x}
  - \Vert u_0^\nu \Vert_{L^2_x}^2 \right) = \int_0^t \langle f , u  \rangle_{L^2_x} \, d\tau.
\end{equation}
However, even though $\| u_0^\nu \|_{L^2_x} \to \| u_0 \|_{L^2_x}$ {\em we cannot conclude} that $u$ satisfies energy balance {\em unless} we have $\Vert u^\nu(t) \Vert^2_{L^2_x} \to \Vert u (t) \Vert^2_{L^2_x}$ a.e. $t \in [0,T]$. As we have already argued, this is equivalent to requiring that $u^\nu$ converge strongly to $u$ in $L^2_t L^2_x$. Therefore, absence of anomalous dissipation does not rule out inviscid dissipation. Nevertheless, it is interesting in its own right to study the absence of anomalous dissipation and we refer the reader to \cite{DeRosaPark2024} for results in this direction.
\end{remark}

%Together,  Propositions \ref{prop:ebal-->conv} and \ref{prop:conv-->ebal} complete the proof of Theorem \ref{thm:ebal}.

\subsection{Improvement from $L^2_tL^2_x$ to $C_tL^2_x$ convergence}
\label{sec:contt}

Putting together  Propositions \ref{prop:ebal-->conv} and \ref{prop:conv-->ebal}, we have shown that Theorem \ref{thm:ebal}\eqref{item1} is equivalent to Theorem \ref{thm:ebal}\eqref{item2}. We will now complete the proof of Theorem \ref{thm:ebal} by showing that these two equivalent statements are also equivalent to Theorem \ref{thm:ebal}\eqref{item3}.

\begin{proposition}
\label{prop:contt}
Let $u$ be a physically realizable solution of the Euler equations \eqref{eq:euler}  with forcing $f\in L^2 ((0,T);L^2 (\T^2))$. Let $u^\nu$ be a physical realization of $u$ and assume that the forcing $f^\nu$ converges strongly to $f$ in $L^2_t L^2_x$.

If $u^\nu \to u$  strongly in $C([0,T];L^2(\T^2))$ then $u^\nu \to u$ converges strongly in $L^2((0,T);L^2(\T^2))$, or, equivalently, $u$ is energy balanced.

Conversely, if either of the following equivalent assertions holds,
\begin{itemize}
\item $u$ is energy balanced, (cf. Theorem \ref{thm:ebal}\eqref{item1}),
\item $u^\nu$ converges to $u$  in $L^2((0,T);L^2(\T^2))$, (cf. Theorem \ref{thm:ebal}\eqref{item2}),
\end{itemize}
then the convergence $u^\nu \to u$ is actually strong in $C([0,T];L^2(\T^2))$, (cf. Theorem \ref{thm:ebal}\eqref{item3}).
\end{proposition}

\begin{proof}

It is immediate that $u^\nu \to u$ strongly in $C([0,T];L^2(\T^2))$ implies $u^\nu \to u$ strongly in $L^2_t L^2_x$. By Propositions \ref{prop:ebal-->conv} and \ref{prop:conv-->ebal} this is equivalent to $u$ being energy balanced.

It remains to show that Theorem \ref{thm:ebal}\eqref{item1} and Theorem \ref{thm:ebal}\eqref{item2} together imply Theorem \ref{thm:ebal}\eqref{item3}. Let $u$ be energy-balanced and $u^\nu \to u$ strongly in $L^2_tL^2_x$.

\textbf{Step 1:}
We start with the observation that, since $u$ is energy-balanced, then $u\in C([0,T];L^2(\T^2))$. To see this first, recall Remark \ref{rem:strongtcont}, in which we observed that the function $t \mapsto \Vert u(t) \Vert_{L^2}$ is continuous. In addition in Lemma \ref{lem:tcont} we showed that $u \in C([0,T];w\text{-}L^2(\T^2))$. It is now immediate that continuity of norms together with continuity in time into $L^2_x$ with the weak topology implies continuity in time into $L^2_x$.

%Let $t_0 \in [0,T]$. Then we have
%\[\Vert u(t) - u(t_0) \Vert_{L^2}^2 = \Vert u(t) \Vert_{L^2}^2 + \Vert u(t_0) \Vert_{L^2}^2 - 2 \int_{\T^2} u(t) \cdot u(t_0) .\]
% Since $u(t_0) \in L^2_x$ it now follows immediately that
%\[\lim_{t \to t_0} \Vert u(t) - u(t_0) \Vert_{L^2}^2 = 0.\]

 We now come to the heart of the proof.

 \textbf{Step 2:}
We claim that $\Vert u^\nu(t) \Vert_{L^2_x}^2 \to \Vert u(t) \Vert_{L^2_x}^2$ uniformly on $[0, T]$.

To see this, we note that, since $u$ is energy balanced and from the energy balance identity \eqref{eq:ns-ebal} for solutions of the Navier-Stokes equations, we have for any $t\in [0,T]$:
\begin{align*}
\Vert u(t) \Vert^2_{L^2_x} - \Vert u^\nu(t) \Vert^2_{L^2_x}
&=
\left\{
\Vert u_0 \Vert^2_{L^2_x} + 2 \int_0^t \langle f, u \rangle_{L^2_x} \, d\tau
\right\}
\\
&\qquad
- \left\{
\Vert u^\nu_0 \Vert^2_{L^2_x} - 2\nu \int_0^t \Vert \omega^\nu \Vert^2_{L^2_x} \, d\tau + 2 \int_0^t \langle f^\nu, u^\nu \rangle_{L^2_x} \, d\tau
\right\}
\\
&=
\Vert u_0 \Vert^2_{L^2_x} - \Vert u^\nu_0 \Vert^2_{L^2_x}
\\
&\qquad + 2 \left\{
\int_0^t \langle f, u \rangle_{L^2_x} \, d\tau
- \int_0^t \langle f^\nu, u^\nu \rangle_{L^2_x} \, d\tau
\right\}
\\
&\qquad +2 \nu \int_0^t \Vert \omega^\nu \Vert^2 \, d\tau.
\end{align*}
Bounding the integal terms on the right-hand side, it follows that
\begin{align*}
\Big|
\Vert u(t) \Vert^2_{L^2_x} - \Vert u^\nu(t) \Vert^2_{L^2_x}
\Big|
&\le
\Big|
\Vert u_0 \Vert^2_{L^2_x} - \Vert u^\nu_0 \Vert^2_{L^2_x}
\Big|
\\
&\qquad + 2\int_0^T \Vert f(\tau) - f^\nu(\tau) \Vert_{L^2_x} \Vert u(\tau) \Vert_{L^2_x} \, d\tau
\\
&\qquad + 2\int_0^T \Vert f^\nu(\tau) \Vert_{L^2_x} \Vert u(\tau) - u^\nu(\tau) \Vert_{L^2_x} \, d\tau
\\
&\qquad +2 \nu \int_0^T \Vert \omega^\nu \Vert^2 \, d\tau
\\
&=: (I) + (II) + (III) + (IV).
\end{align*}
Since the right-hand side is independent of $t$, we conclude that
\[
\sup_{t\in [0,T]} \Big|
\Vert u(t) \Vert^2_{L^2_x} - \Vert u^\nu(t) \Vert^2_{L^2_x}
\Big|
\le (I) + (II) + (III) + (IV).
\]
It remains to show that the terms on the right-hand side converge to $0$ as $\nu \to 0$.

We note that $(I)$ converges to $0$ as $\nu \to 0$, since by assumption, $u^\nu_0 \to u_0$ strongly in $L^2_x$, and hence $\Vert u^\nu_0 \Vert_{L^2_x} \to \Vert u_0 \Vert_{L^2_x}$. Next, we can bound
\[
(II) \le \Vert f- f^\nu \Vert_{L^2_tL^2_x} \Vert u \Vert_{L^2_t L^2_x} \to 0,
\]
by assumption on $f^\nu \to f$ in $L^2_t L^2_x$. Furthermore, since $\Vert f^\nu \Vert_{L^2_tL^2_x}\le M$ is uniformly bounded, and since we also assume that $u^\nu \to u$ in $L^2_tL^2_x$, we similarly conclude that $(III)\to 0$. Finally, since $u^\nu \to u$ in $L^2_tL^2_x$, the convergence $(IV)\to 0$, as $\nu \to 0$, follows from Corollary \ref{cor:vanishing-enstrophy}.

\textbf{Step 3:} We finally claim that
\[
\sup_{t\in [0,T]}
\Vert u(t) - u^\nu(t) \Vert_{L^2_x} \to 0,\quad \text{ as } \nu \to 0.
\]

We argue by contradiction. If this is not the case, then there exists a convergent sequence $t_n \to t \in [0,T]$ and a sequence $\nu_n\to 0$, such that
\[
\Vert u(t_n) - u^{\nu_n}(t_n) \Vert_{L^2_x} \ge \epsilon_0 > 0.
\]
But, by the uniform convergence of $\Vert u^{\nu_n}(\slot) \Vert_{L^2_x} \to \Vert u(\slot) \Vert_{L^2_x}$ in $C([0,T])$, it follows that $\Vert u^{\nu_n}(t_n) \Vert_{L^2_x} \to \Vert u(t)\Vert_{L^2_x}$. From the convergence $u^{\nu_n} \to u$ in $C([0,T];w-L^2_x)$, see the proof of Lemma \ref{lem:tcont}, it follows that $u^{\nu_n}(t_n) \weaklyto u(t)$ in $L^2_x$. Since norm convergence and weak convergence imply strong convergence, we conclude that
\[
\lim_{n\to \infty} \Vert u^{\nu_n}(t_n) - u(t) \Vert_{L^2_x} = 0,
\]
Owing to the fact that $u\in C([0,T];L^2_x)$ (cf. Step 0 of this proof), this leads to a contradiction:
\begin{align*}
0
&< \epsilon_0 \le
\limsup_{n\to \infty} \Vert u(t_n) - u^{\nu_n}(t_n) \Vert_{L^2_x}
\\
&\le
\limsup_{n\to \infty} \Vert u^{\nu_n}(t_n) - u(t) \Vert_{L^2_x}
\\
&\qquad
+
\limsup_{n\to \infty} \Vert u(t) - u(t_n) \Vert_{L^2_x}
 = 0.
\end{align*}
Thus, we must have that $u^{\nu_n} \to u$ in $C([0,T];L^2_x)$.
\end{proof}

Putting together Propositions \ref{prop:ebal-->conv}, \ref{prop:conv-->ebal} and \ref{prop:contt} we have completed the proof of Theorem \ref{thm:ebal}.

\section{Examples}
\label{sec:examples}

Theorem \ref{thm:ebal}  provides necessary and sufficient conditions for a physically realizable solution $u$ to be energy balanced. In the present section, we focus on specific classes of initial data and forcing for which these conditions are satisfied. More precisely, we will study solutions whose vorticity belongs to rearrangement invariant spaces, including  $L^p$ ($p>1$), the Orlicz spaces $L\log(L)^\alpha$ ($\alpha > 1/2$), and the (modified) Lorentz spaces $L^{(1,q)}$ ($1\le q\le 2$).

Besides exhibiting instances to which Theorem \ref{thm:ebal} applies, the goals of this section are two-fold:
\begin{itemize}
\item[(i)] we show how the present work extends the main result of \cite[Thm. 2.4]{LN2022}, where energy balance was shown for physically realizable solutions with $L^p$-control on the vorticity for $p>1$;

\item[(ii)] we fill a gap in the proof of \cite[Corollary 2.13]{LMPP2021a}, where it was asserted that certain bounds on decreasing rearrangements are preserved by the solution operator of the Navier-Stokes equations, in the absence of an external force. A detailed proof of this assertion has, so far, been missing from the literature.
\end{itemize}

In Section \ref{sec:Lp} we connect our main results with the recent work \cite{LN2022} on $L^p$ vorticity control. Section \ref{sec:rinv} contains basic definitions on rearrangement invariant spaces; we also include the statements of our main results on {\it a priori} vorticity control for solutions of the Navier-Stokes equations. In Section \ref{sec:Lorentz} we extend the discussion of energy balance from the rearrangement-invariant $L^p$-spaces ($p>1$) to more general rearrangement-invariant spaces, including Orlicz and Lorentz spaces.

\subsection{Solutions with vorticity in $L^p$, $p>1$}
\label{sec:Lp}

Proposition \ref{prop:conv-->ebal}  implies the following stronger version of \cite[Thm. 2.4]{LN2022}:
\begin{corollary}
  \label{cor:Lp}
  Let $u$ be a physically realizable solution of the incompressible Euler equations \eqref{eq:euler} with external forcing $f \in L^1((0,T);L^2(\T^2))$. Consider a physical realization $u^\nu $ of $u$, with viscosity $\nu > 0$ and forcing $f^\nu $, as in Definition \ref{def:physrealwsoln}. Suppose, in addition, that for some $p>1$:
  \begin{enumerate}
  \item $\omega_0 \equiv \curl u_0  \in L^p(\T^2)$,
  \item $\omega_0^\nu \equiv \curl u_0^\nu  \to \omega_0$ strongly in $L^p(\T^2)$,
  \item $g^\nu = \curl f^\nu $ is bounded in $L^2((0,T);L^p(\T^2))$.
  \end{enumerate}
  Then $u$ is an energy balanced weak solution.
\end{corollary}

\begin{remark}
  To derive the corresponding result in \cite[Thm. 2.4]{LN2022}  the authors assume that $g^\nu$ is uniformly bounded in $L^1((0,T);L^p(\T^2)) \cap L^\infty((0,T);L^2(\T^2))$. We note that,   in the most relevant range $1<p < 2$, this is  strictly stronger than assumption (3) of Corollary \ref{cor:Lp}.
\end{remark}

\begin{proof}
We begin by observing that, from the hypothesis that $g^\nu$ is bounded in $L^2_tL^p_x$, we have, using elliptic regularity and the Poincar\'e inequality, that $f^\nu$ is bounded in $L^2_t W^{1,p}_x$. Therefore, since $W^{1,p}(\T^2)$ is continuously embedded in $L^2(\T^2)$ for $p\geq 1$, we obtain that $\{f^\nu \} \subset L^2_tL^2_x$ is uniformly bounded. Thus it is precompact in $L^2_t L^2_x$ with the weak topology. Since $f^\nu \rightharpoonup f$ in $L^1_t L^2_x$ it follows that $f^\nu$ converges weakly to $f$ in $L^2_t L^2_x$ as well.

Next, we make a small adaptation in the proof of  \cite[Lemma 3.1]{LN2022} to show that  the assumptions of Corollary \ref{cor:Lp} imply that $u^\nu$ converges strongly to $u$ in $C([0,T];L^2(\T^2))$. For convenience of the reader, recall the estimate \cite[(3.3)]{LN2022}:
\[\Vert \omega^\nu (t) \Vert_{L^p} \leq C \left( \Vert \omega^\nu_0 \Vert_{L^p} + \int_0^T \Vert g^\nu (s) \Vert_{L^p} ds\right).\]
Using the Cauchy-Schwarz inequality in the integral term above we conclude that $\omega^\nu$ is bounded in $L^\infty_t L^p_x$, uniformly with respect to $\nu$. Thus, as in \cite[Lemma 3.1]{LN2022}, we use elliptic regularity and the Poincar\'e inequality to conclude that $u^\nu$ is bounded in $L^\infty_t W^{1,p}_x$. We then further deduce from the PDE \eqref{eq:ns}, that $\partial_t u^\nu$ is bounded in $L^2_t H^{-M}_x$ for some
$M \in \mathbb{N}$. Since the embedding of $W^{1,p}(\T^2)$ into $L^2(\T^2)$ is compact for $p>1$, it follows by the Aubin-Lions-Simons Lemma, see \cite[Theorem II.5.16]{BoyerFabrie2013}, that $\{u^\nu\}$ is compact in $C([0,T];L^2(\T^2))$. Passing to subsequences as needed without relabeling we find $u^\nu \to v$  strongly in $C([0,T];L^2(\T^2))$. By hypothesis we already know that $u^\nu \rightharpoonup u$ weak-$\ast$ $L^\infty_t L^2_x$. Therefore $v=u$ and the whole family converges strongly in $C([0,T];L^2(\T^2))$.

The desired result is now an immediate corollary of Proposition \ref{prop:conv-->ebal}.
\end{proof}

\subsection{{\it A priori} estimates in rearrangement-invariant spaces}
\label{sec:rinv}

Aiming to generalize the result of the previous section, which pertains to solutions with $L^p$-vorticity control, we consider solutions with vorticity in more general rearrangement invariant spaces in the present section. In order to treat such spaces we will need the following definition.

\begin{definition} \label{rearrmaxfctn}
Let $f \in L^1(\T^2)$. The rearrangement invariant maximal function for $f$ is
\[
\cM_s(f) := \sup \left\{\int_E |f(x)| \, dx \, \Big| \, E \subset \T^2, E \text{ measurable, }|E|=s  \right\},
\]
defined for $0\le s \le |\T^2| = (2\pi)^2$.
\end{definition}

\begin{remark} \label{fstarstar}
We mention in passing that, since the torus $\T^2$ with the Haar measure is a strongly resonant measure space, the function $\cM_s(f)$  defined above corresponds to $s f^{\ast\ast}(s)$, where $f^{\ast\ast}$ is the standard maximal function of the non-increasing rearrangement function $f^\ast$; see \cite[Chapter 2, Section 3]{BennettSharpley} for additional information.
\end{remark}

 In this section, we prove the following {\it a priori} estimate for solutions of the Navier-Stokes equations:

\begin{proposition}
  \label{prop:rearrange}
  Let $u^\nu \in L^\infty((0,T);L^2(\T^2)$ be the unique solution of the Navier-Stokes equations \eqref{eq:ns}, with  forcing $f^\nu \in L^1((0,T);L^2(\T^2))$ and  initial data $u^\nu_0\in L^2(\T^2)$, both assumed to be divergence-free. Assume that $\omega_0^\nu = \curl u^\nu_0  \in L^1 (\T^2)$, and $g^\nu = \curl f^\nu  \in L^1 ((0,T);L^1 (\T^2))$. Then for any $t\in [0,T]$, we have the following {\it a priori} estimate:
  \begin{align}
    \label{eq:rearrange}
  \cM_s(\omega^\nu(t))
  \le
  \cM_s(\omega^\nu_0)
  +
  \int_0^t \cM_s(g^\nu(\tau)) \, d\tau.
  \end{align}
\end{proposition}

Our proof of this fact is based on operator splitting: The idea is to approximate the solution of the vorticity form of the Navier-Stokes equations by a composition of (small) time-steps for the forced heat equation and a transport PDE, respectively. The necessary {\it a priori} bounds for the heat and the transport equations are readily derived based on explicit solution formulae. If the operator splitting scheme converges to the solution operator of the Navier-Stokes equations in a suitable norm, then the necessary bounds for the Navier-Stokes equations can be deduced from a limiting argument.

One  difficulty with this approach is the potential lack of smoothness of the underlying solution; in the absence of such smoothness, the operator splitting scheme is not known to converge. We circumvent this difficulty via an additional mollification argument: we establish the required bounds for solutions of the mollified system, and extend these {\it a priori} bounds to rough solutions by a limiting argument. The details of this argument are contained in Sections \ref{sec:operator}--\ref{sec:pf-rearrange}.

Section \ref{sec:operator} discusses basic {\it a priori} estimates for the heat and transport PDEs involved in the operator splitting scheme. Section \ref{sec:split-approx} defines the operator splitting approximant, and provides relevant estimates for this approximant. Section \ref{sec:pf-rearrange} combines these results to prove Proposition \ref{prop:rearrange}.

\subsubsection{Split estimates}
\label{sec:operator}

Given initial data $\omega_0^\nu \in C^\infty(\T^2)$ and forcing $g^\nu \in C^\infty(\T^2\times [0,T])$, standard well-posedness theory of the Navier-Stokes equations implies that there exists a unique smooth solution $\omega^\nu \in C^\infty(\T^2\times [0,T])$ of the vorticity formulation of the Navier-Stokes equations. Given such a smooth solution, our aim is to derive estimates on $\cM_s(\omega^\nu)$ through operator splitting. To this end, we decompose the vorticity equation:
\begin{align}
\label{eq:1}
\partial_t \omega^\nu = \underbrace{-u^\nu \cdot \nabla \omega^\nu}_{(E)} + \underbrace{\nu \Delta \omega^\nu + g^\nu}_{(H)},
\end{align}
according to the two terms $(E)$ and $(H)$ on the right-hand side.
\bigskip

In the following we assume that the smooth initial data $\omega_0^\nu$ and smooth forcing $g^\nu$ are fixed. We denote by $\omega^\nu \in C^\infty(\T^2\times [0,T])$ the corresponding solution, and we denote by $u^\nu \in C^\infty(\T^2\times[0,T])$ the divergence-free velocity field satisfying $\omega^\nu = \curl u^\nu$.

\bigskip

As indicated in \eqref{eq:1}, the vorticity equation can be split into a transport equation and a forced heat equation. We next introduce the corresponding solution operators.
\bigskip

\paragraph*{\textbf{Transport equation.}}
Given $t_0, t \in [0,T]$, $t\ge t_0$, we denote by $\beta_0 \mapsto E(t;t_0) \beta_0$ the solution operator associated with the transport PDE
\begin{align}
\label{eq:E0}
\partial_t \beta(t) + u^\nu(t) \cdot \nabla \beta(t) = 0,
\qquad
\beta(t_0) = \beta_0,
\tag{E}
\end{align}
so that $E(t;t_0)\beta_0 := \beta(t)$. In \eqref{eq:E0}, $u^\nu(t) = u^\nu(\slot,t)$ denotes the velocity field obtained by solving the Navier-Stokes equations with the fixed initial vorticity $\omega^\nu_0$ and the fixed vorticity forcing $g^\nu$. The following proposition gives a simple {\it a priori} estimate on rearrangements under $E(t;t_0)$:

\begin{proposition}
\label{prop:Erearrange}
Let $t_0,t\in [0,T]$, $t\ge t_0$. If $u^\nu$ is smooth, and if $\beta (t) = E(t,t_0)\beta_0$ is a solution of the transport PDE \eqref{eq:E0} with data $\beta \in L^{1}(\T^2)$  then, for all $s \ge 0$, we have
\[
\cM_s(\beta(t))
=
\cM_s(\beta_0).
\]
\end{proposition}

\begin{proof}
Let $\phi : \T^2 \times [t_0,t] \to \T^2$ denote the  flow-map of $u^\nu$, i.e.
\[
\left\{
\begin{array}{ll}
\displaystyle{\frac{d\phi}{d\tau}} (x,\tau)  = u^\nu(\phi (x,\tau),\tau), & \tau \in [t_0,t], \\ & \\
\phi(x,t_0)  = x, & x \in \T^2.
\end{array}
\right.
\]
We recall that $\beta(t) = \beta_0\circ [\phi(\cdot,t)]^{-1}$. Since $u^\nu$ is divergence-free, $\phi(\cdot,t)$ is measure-preserving, i.e. $|\phi(\cdot,t)^{-1}(E)|=|E|$. In particular, it follows that
\begin{align*}
\cM_s(\beta(t))
&=
\sup_{|E|=s} \int_E |\beta(x,t)| \, dx
=
\sup_{|E|=s} \int_{\phi (\cdot, t)^{-1}(E)} |\beta_0(x)| \, dx
\\
&=
\sup_{|E|=s} \int_E |\beta_0(x)| \, dx
=
\cM_s(\beta_0).
\end{align*}
\end{proof}

\bigskip
\paragraph*{\textbf{Heat equation.}}
Let $t, t_0\in [0,T]$ with $t\ge t_0$. We denote by $\beta_0 \mapsto H(t;t_0)\beta_0$ the solution operator associated with the forced heat equation, i.e. we set $H(t;t_0)\beta_0 := \beta(t)$, where $\beta$ solves
\begin{align}
\label{eq:H0}
\partial_t \beta(t) = \nu \Delta \beta(t) + g^\nu(t),
\qquad
\beta(t_0) = \beta_0.
\tag{H}
\end{align}
In the following, we derive simple {\it a priori} estimates on rearrangements under $H(t;t_0)$. We first consider the action of the heat kernel.
\begin{proposition}
  \label{prop:Gast}
  For $t > 0$, let $G_t: \T^2 \to \R$ denote the $\nu$-heat kernel on the 2-dimensional torus; more precisely, let
  \[
  G_t(x) := \frac{1}{4\pi \nu t} \sum_{k\in \Z^2} e^{-\frac{(x-2\pi k)^2}{4\nu t}}.
  \]
  Then for any $\beta \in C^\infty(\T^2)$ and $s\ge 0$, we have
  \[
  \cM_s\left( G_t \ast \beta \right)
  \le
  \cM_s(\beta).
  \]
\end{proposition}

\begin{proof}
  We note that $\Vert G_t \Vert_{L^1} = 1$, and hence
  \begin{align*}
  \cM_s\left( G_t \ast \beta \right)
  &=
  \sup_{|E|=s}
  \int_E |G_t \ast \beta| \, dx
  \\
  &\le
  \sup_{|E|=s}
  \int_{\T^2} 1_E(x) \int_{\T^2} G_t(y) |\beta(x-y)| \, dy \, dx
  \\
  &=  \sup_{|E|=s}
  \int_{\T^2} G_t(y) \left( \int_{\T^2} 1_E(x) |\beta(x-y)| \, dx\right)  \, dy
  \\
  &\le
  \sup_{|E|=s}\Vert G_t \Vert_{L^1}
  \;
  \sup_{y\in \T^2}
  \int_E |\beta(x-y)| \, dx
  \\
  &= \sup_{|E|=s}
  \int_E |\beta(y)| \, dy
  \\
  &= \cM_s(\beta).
  \end{align*}
\end{proof}

As a consequence of the last proposition, we obtain
\begin{proposition}
\label{prop:Hrearrange}
  Let $t,t_0\in [0,T]$ be given, such that $t\ge t_0$. Let $\beta_0 \in C^\infty(\T^2)$, $g^\nu\in C^\infty(\T^2 \times [0,T])$, and let
  \[
  \beta(t) = H(t;t_0)\beta_0 \in C^\infty(\T^2 \times [0,T]),
  \]
   be the solution of the forced heat equation \eqref{eq:H0}, i.e.
  \[
  \partial_t \beta(t) = \nu \Delta \beta(t) + g^\nu(t), \qquad \beta(t_0) = \beta_0.
  \]
  Then for any $s\ge 0$:
  \[
  \cM_s (\beta(t))
  \le
  \cM_s (\beta_0)
  +
  \int_{t_0}^t \cM_s (g^\nu(\tau)) \, d\tau.
  \]
\end{proposition}

\begin{proof}
  By Duhamel's formula, the solution of the forced heat equation is given by
  \[
  \beta(\slot, t) = G_{t-t_0} \ast \beta_0 + \int_{t_0}^{t} G_{t-\tau} \ast g^\nu(\slot, \tau) \, d\tau.
  \]
  Thus,
  \begin{align*}
  \cM_s (\beta(t))
  &=
  \sup_{|E|=s} \int_E |\beta(\slot, t)| \, dx
  \\
  &\le
  \sup_{|E|=s} \int_E |G_{t-t_0} \ast \beta_0| \, dx
  +
  \int_{t_0}^t \left( \sup_{|E|=s} \int_E |G_{t-\tau} \ast g^\nu(\slot,\tau)| \, dx \right) \, d\tau
  \\
  &=
  \cM_s(G_{t-t_0} \ast \beta_0)
  +
  \int_{t_0}^t \cM_s (G_{t-\tau}\ast g^\nu(\tau)) \, d\tau.
  \\
  &\le
  \cM_s (\beta_0)
  +
  \int_{t_0}^t \cM_s (g^\nu(\tau)) \, d\tau,
  \end{align*}
 where the last inequality follows from Proposition \ref{prop:Gast}.
\end{proof}

\subsubsection{Operator splitting approximation}
\label{sec:split-approx}

In view of the definitions of $E(t;t_0)$, as the solution operator of the transport PDE \eqref{eq:E0} for $t\ge t_0$, and $H(t;t_0)$, as the solution operator of the forced heat equation \eqref{eq:H0} for $t\ge t_0$, we now define an ``operator splitting'' approximation of $\omega^\nu$ recursively as follows: Fix a time-step $\Delta t$, and for $n\in \N_0$ define $t_n = n \Delta t$. Given initial data $\omega^\nu_0$, we set $\omega^{\nu,\Delta}_0(t_0) := \omega^\nu_0$ at $t=t_0=0$. Then, we recursively define
\begin{align}
  \label{eq:osplit}
  \left\{
  \begin{aligned}
  \omega^{\nu,\Delta}_{n+1/2}(t) &:= E(t;t_{n}) \omega^{\nu,\Delta}_{n}(t_{n}), \\
  \omega^{\nu,\Delta}_{n+1}(t) &:= H(t;t_{n}) \omega^{\nu,\Delta}_{n+1/2}(t),
  \end{aligned}
  \right.
  \qquad  \text{for }t \in (t_{n},t_{n+1}].
\end{align}
We finally define $\omega^{\nu,\Delta}$ piecewise in time, by setting
\begin{align}
\label{eq:osplit-all}
\omega^{\nu,\Delta}(t) := \omega^{\nu,\Delta}_n(t), \quad \text{for }
t\in (t_{n-1},t_n].
\end{align}

Before providing formal motivation for \eqref{eq:osplit} (cf. the next remark), we would like to point out that
\[
\omega^{\nu,\Delta}_{n+1}(t) = H(t;t_{n}) \omega^{\nu,\Delta}_{n+1/2}(t),
\]
depends on time $t$ through both the solution operator $H(t;t_n)$ and additionally through the  data $\omega^{\nu,\Delta}_{n+1/2}(t)$ to which this solution operator is applied. To avoid confusion, we point out that we could have equivalently defined $\omega^{\nu,\Delta}_{n+1}(t)$ in two steps, by first setting for $\tau_1,\tau_2 \in [t_n,t_{n+1}]$:
\begin{align}
\label{eq:h}
h(\tau_1,\tau_2) := H(\tau_1; t_n) \omega^{\nu,\Delta}_{n+1/2}(\tau_2),
\end{align}
and then setting $\omega^{\nu,\Delta}_{n+1}(t) = h(t,t)$.

\begin{remark}
To motivate the above definition, we note that upon formally expanding in time and neglecting terms $O(|t-t_n|^2)$, we have
\begin{align}
\label{eq:formal1}
\begin{aligned}
\omega^{\nu,\Delta}_{n+1/2}(t)
&= E(t;t_n) \omega^{\nu,\Delta}_{n}(t_{n})
\\
&\approx \omega^{\nu,\Delta}_{n}(t_{n}) - [t-t_n] \, u^\nu(t_n) \cdot \nabla  \omega^{\nu,\Delta}_{n}(t_{n}),
\end{aligned}
\end{align}
and
\begin{align}
\label{eq:formal2}
\begin{aligned}
\omega^{\nu,\Delta}_{n+1}(t)
&= H(t;t_n) \omega^{\nu,\Delta}_{n+1/2}(t)
\\
&\approx
\omega^{\nu,\Delta}_{n+1/2}(t) + [t-t_n] \left\{ g^\nu(t) + \nu \Delta \omega^{\nu,\Delta}_{n+1/2}(t) \right\}.
\end{aligned}
\end{align}
Inserting \eqref{eq:formal1} into \eqref{eq:formal2}, rearranging and retaining only lowest order terms in $[t-t_n]$, we find
\[
\omega^{\nu,\Delta}_{n+1}(t)
\approx
\omega^{\nu,\Delta}_{n}(t_n)
+ [t-t_n] \left\{ -u^\nu(t_n) \cdot \nabla  \omega^{\nu,\Delta}_{n}(t_{n})
+ g^\nu(t)
+ \nu \Delta \omega^{\nu,\Delta}_{n}(t_n)  \right\}.
\]
Equivalently, upon rearranging and stating this equation in terms of $\omega^{\nu,\Delta}(t)$, we find for $t\in [t_n,t_{n+1}]$,
\[
\frac{\omega^{\nu,\Delta}(t) - \omega^{\nu,\Delta}(t_n)}{t-t_n}
= -u^\nu(t_n) \cdot \nabla  \omega^{\nu,\Delta}(t_{n})
+ g^\nu(t)
+ \nu \Delta \omega^{\nu,\Delta}(t_n).
\]
Expanding the terms in this equation once more around $t$, we have
\begin{align*}
\frac{\omega^{\nu,\Delta}(t) - \omega^{\nu,\Delta}(t_n)}{t-t_n}
&= \partial_t \omega^{\nu,\Delta}(t) + O(\Delta t),
\\
-u^\nu(t_n) \cdot \nabla  \omega^{\nu,\Delta}(t_{n})
&= -u^\nu(t)\cdot \nabla  \omega^{\nu,\Delta}(t) + O(\Delta t),
\\
\nu \Delta \omega^{\nu,\Delta}(t_n)
&= \nu \Delta \omega^{\nu,\Delta}(t) + O(\Delta t).
\end{align*}
Thus, formally up to terms of order $O(\Delta t)$, the function $\omega^{\nu,\Delta}(t)$ defined by \eqref{eq:osplit} solves the equation,
\[
\partial_t \omega^{\nu,\Delta}(t) = -u^\nu(t) \cdot \nabla \omega^{\nu,\Delta}(t) + g^\nu(t) + \nu \Delta \omega^{\nu,\Delta}(t)
+ O(\Delta t).
\]
This provides the formal justification for our definition of the splitting approximant $\omega^{\nu,\Delta}$. To make this precise, a detailed analysis of the $O(\Delta t)$ correction term is required. The detailed derivation will be provided in Appendix \ref{app:adv-diff}.
\end{remark}

  Combining Propositions \ref{prop:Erearrange} and \ref{prop:Hrearrange} we obtain the following {\it a priori} control on rearrangements for the operator splitting approximant $\omega^{\nu,\Delta}$:

\begin{lemma}
\label{lem:rearrange}
Let $\omega^\nu_0 \in C^\infty(\T^2)$ be initial data for the Navier-Stokes equations in vorticity formulation with forcing $g^\nu \in C^\infty(\T^2\times [0,T])$. Let $\omega^{\nu,\Delta}$ be the operator splitting approximation \eqref{eq:osplit} for a given time-step $\Delta t>0$. Then for any time $t\in [0,T]$ and for any $s\ge 0$, we have the following estimate
\begin{align}
\label{eq:rearrange}
\cM_s (\omega^{\nu,\Delta}(t))
\le
\cM_s (\omega^\nu_0)
+
\int_0^t \cM_s (g^\nu(\tau)) \, d\tau.
\end{align}
\end{lemma}

\begin{proof}
  Given $t\in [0,T]$, we can choose $n\in \N$, such that $t\in [t_n,t_{n+1}]$. By definition, we have $\omega^{\nu,\Delta}(t) = \omega^{\nu,\Delta}_n(t)$. It thus suffices to prove that for $n\in \N$, we have
  \begin{align}
    \label{eq:reduce}
  \cM_s (\omega^{\nu,\Delta}_n(t))
  \le
  \cM_s (\omega^\nu_0)
  +
  \int_0^t \cM_s (g^\nu(\tau)) \, d\tau,
  \quad \forall t\in [t_n,t_{n+1}].
  \end{align}
  We prove \eqref{eq:reduce} by induction on $n$. The claim is trivially true for $n=0$. For the induction step and fixed $t \in [t_{n},t_{n+1}]$, we recall that $\omega^{\nu,\Delta}_{n+1/2}(t) = E(t;t_n) {\omega}_n^{\nu,\Delta}(t_n)$, and ${\omega}^{\nu,\Delta}_{n+1}(t) = H(t;t_n) \omega^{\nu,\Delta}_{n+1/2}(t)$. Since the advecting velocity field $u^\nu$ is smooth for solutions of 2D Navier-Stokes with smooth initial data and forcing, Proposition \ref{prop:Hrearrange} implies that
\begin{align*}
 \cM_s({\omega}^{\nu,\Delta}_{n+1}(t))
 &=
 \cM_s\left(H(t;t_n) \omega_{n+1/2}^{\nu,\Delta}(t)\right)
 \\
 &\le
  \cM_s (\omega_{n+1/2}^{\nu,\Delta}(t)) + \int_{t_n}^t \cM_s(g^\nu(\tau)) \, d\tau.
\end{align*}
Furthermore, since $\omega_{n+1/2}^{\nu,\Delta}(t) = E(t;t_n)\omega_{n}^{\nu,\Delta}(t_n)$, by Proposition \ref{prop:Erearrange} we have
\[
  \cM_s (\omega_{n+1/2}^{\nu,\Delta}(t))
  =
  \cM_s ({\omega}_{n}^{\nu,\Delta}(t_n)).
\]

By the induction hypothesis, we have
\[
\cM_s ({\omega}_{n}^{\nu,\Delta}(t_n))
\le
\cM_s (\omega^\nu_0)
+
\int_0^{t_n} \cM_s (g^\nu(\tau)) \, d\tau.
\]
Combining these estimates yields \eqref{eq:reduce}.
\end{proof}

The previous result provides {\it a priori} bounds on the operator splitting approximant $\omega^{\nu,\Delta}$. The next result shows that the operator splitting approximant $\omega^{\nu,\Delta} \to \omega^\nu$ converges to the solution, provided that the underlying forcing and solution are sufficiently regular:
  \begin{proposition}
    \label{prop:osplit-conv}
    Let $u^\nu_0\in C^\infty(\T^2)$ be smooth divergence-free initial data for the incompressible Navier-Stokes equations \eqref{eq:ns}. Assume that that the forcing $f^\nu \in C^\infty(\T^2\times [0,T])$ is smooth. Let $u^\nu \in C^\infty(\T^2\times [0,T])$ be the unique smooth solution with this data. Then the operator splitting approximant $\omega^{\nu,\Delta}$ defined by \eqref{eq:osplit} converges to $\omega^\nu$; more precisely, we have
    \[
    \lim_{\Delta t \to 0} \Vert \omega^{\nu} - \omega^{\nu,\Delta} \Vert_{L^\infty_t L^2_x} = 0.
    \]
  \end{proposition}

  \begin{remark}
  The last proposition implies in particular that $\omega^{\nu,\Delta}(\slot,t) \to \omega^\nu(\slot,t)$ converges in $L^1_x$ for any $t \in [0,T]$.
  \end{remark}

  \begin{proof}[Proof of Proposition \ref{prop:osplit-conv}]
    This is a direct consequence of the convergence result for operator splitting applied to the forced advection-diffusion PDE $\partial_t \beta + U\cdot \nabla \beta = \nu \Delta \beta + g$, which we have included in Appendix \ref{app:adv-diff} for completeness; more precisely, we invoke Proposition \ref{prop:splitconv} with $\beta := \omega^\nu$, $\beta_0 := \omega^\nu_0$, $U := u^\nu$, $g := g^\nu = \curl(f^\nu)$. This yields the claim. We emphasize that the convergence rate in this Proposition depends on certain $C^k$-norms of $\omega_0^\nu$, $u^\nu$, $f^\nu$ and on $\nu > 0$, and hence this result applies only to smooth solutions of the Navier-Stokes equations.
  \end{proof}

Given the convergence of operator splitting, Proposition \ref{prop:osplit-conv}, we next aim to derive an {\it a priori} estimate for the rearrangement invariant vorticity maximal functions $\cM_s(\omega^\nu)$, where $\omega^\nu$  is a solution of the vorticity form of the Navier-Stokes equations. To this end, we will need the following simple lemma:

\begin{lemma}
\label{lem:rearrconv}
  If $\omega^\Delta \to \omega$ converges in $L^1(\T^2)$, then for any $s \ge 0$, we have
  \[
  \cM_s( \omega )
  =
  \lim_{\Delta} \cM_s (\omega^\Delta).
  \]
\end{lemma}

\begin{proof}
The convergence $\omega^\Delta \to \omega$ in $L^1$ implies the convergence of the rearrangements, since (see e.g. Thm. 1.D of \cite{talenti})
\[
|\cM_s(\omega) - \cM_s(\omega^\Delta)|
\le
\Vert \omega - \omega^\Delta \Vert_{L^1(\T^2)} \to 0.
\]
\end{proof}

\subsubsection{Proof of Proposition \ref{prop:rearrange}}
\label{sec:pf-rearrange}

\begin{proof}
Based on the results of the last sections, we finally prove that for solutions $u^\nu \in L^\infty((0,T);L^2(\T^2))\cap L^2((0,T);H^1(\T^2))$ of the Navier-Stokes equations \eqref{eq:ns} with additional vorticity control $\omega_0^\nu = \curl(u^\nu_0) \in L^1(\T^2)$ and $g^\nu = \curl(f^\nu) \in L^1((0,T);L^1(\T^2))$, we have the following {\it a priori} bound on the vorticity maximal function
\[
\cM_s(\omega^\nu(t))
\le
\cM_s(\omega^\nu_0) + \int_0^t \cM_s(g^\nu(\tau)) \, d\tau.
\]
We aim to deduce this estimate from the corresponding estimate for suitable operator splitting approximants, equation \eqref{eq:reduce}.

To this end, we denote by $\omega^{\nu,\epsilon}_0$ the mollification of the initial data with a smooth mollifier (e.g. applying the heat kernel for time $\epsilon$), and we denote by $g^{\nu,\epsilon}$ the mollification of $g^\nu$ in both space and time (where we extend $g^\nu$ by zero for $t\notin [0,T]$). Let $\omega^{\nu,\epsilon}$ denote the solution of the corresponding vorticity equation
\[
\partial_t \omega^{\nu,\epsilon} + u^{\nu,\epsilon} \cdot \nabla \omega^{\nu,\epsilon} = \nu \Delta \omega^{\nu,\epsilon} + g^{\nu,\epsilon},
\quad
\omega^{\nu,\epsilon}(t=0) = \omega^{\nu,\epsilon}_0.
\]
Finally, let $\omega^{\nu,\epsilon,\Delta}$ denote the operator splitting approximant of $\omega^{\nu,\epsilon}$ for a time-step $\Delta t$. By Lemma \ref{lem:rearrange}, we have
\[
\cM_s(\omega^{\nu,\epsilon,\Delta}(t))
\le
\cM_s(\omega^{\nu,\epsilon}_0)
+
\int_0^t \cM_s(g^{\nu,\epsilon}(\tau)) \, d\tau.
\]
It follows from Proposition \ref{prop:osplit-conv} that $\omega^{\nu,\epsilon,\Delta} {\to} \omega^{\nu,\epsilon}$ in $L^\infty_tL^1_x$ as $\Delta t\to 0$. Lemma \ref{lem:rearrconv} thus implies that
\[
\cM_s(\omega^{\nu,\epsilon}(t))
\le
\cM_s(\omega^{\nu,\epsilon}_0)
+
\int_0^t \cM_s(g^{\nu,\epsilon}(\tau)) \, d\tau.
\]
Next, we note that mollification decreases the value of the vorticity maximal function (cf. the proof of Proposition \ref{prop:Gast}), so that
\begin{gather}
\label{eq:msbd11}
\begin{aligned}
\cM_s(\omega^{\nu,\epsilon}(t))
&\le
\cM_s(\omega^{\nu,\epsilon}_0)
+
\int_0^t \cM_s(g^{\nu,\epsilon}(\tau)) \, d\tau
\\
&\le
\cM_s(\omega^{\nu}_0)
+
\int_0^t \cM_s(g^{\nu}(\tau)) \, d\tau
\end{aligned}
\end{gather}
is uniformly bounded for any $\epsilon > 0$. Finally, we note that the last estimate implies that, for each $t \in [0,T]$, the family $\set{\omega^{\nu,\epsilon}(t)}{\epsilon \in (0,1]}$ is weakly compact $L^1(\T^2)$ by the Dunford-Pettis theorem. Furthermore, $\omega^{\nu,\epsilon} \to \omega^{\nu}$ in $C([0,T];w\text{-}H^{-1}(\T^2))$, since $u^{\nu,\epsilon}\to u^\nu$ in $C([0,T];w\text{-}L^2(\T^2))$ by classical well-posedness of the Navier-Stokes equations in dimension two. Therefore, it follows that the only weak $L^1$-limit point of $\omega^{\nu,\epsilon}(t)$ is $\omega^\nu(t)$, and hence
$\omega^{\nu,\epsilon} \to \omega^{\nu}$  in $C([0,T];w\text{-}L^1(\T^2))$, as $\epsilon \to 0$. It is easy to see that $\cM_s(\slot)$  is weakly-$L^1$ lower semi-continuous so, using \eqref{eq:msbd11}, this gives
\[
\cM_s(\omega^{\nu}(t))
\le
\liminf_{\epsilon \to 0}
\cM_s(\omega^{\nu,\epsilon}(t))
\le
\cM_s(\omega^{\nu}_0)
+
\int_0^t \cM_s(g^{\nu}(\tau)) \, d\tau.
\]
This concludes the proof.
\end{proof}

\subsection{The Lorentz space $L^{(1,q)}(\T^2)$}
\label{sec:Lorentz}

For $1 \le q <\infty$, the rearrangement-invariant Lorentz space $L^{(1,q)}(\T^2)$ is defined as the space
\[
L^{(1,q)}(\T^2)
=
\set{\omega \in L^1(\T^2)}{\Vert \omega \Vert_{L^{(1,q)}} < \infty},
\]
with norm
\begin{align*}
\Vert \omega \Vert_{L^{(1,q)}}
&=
\left(
\int_0^{|\T^2|} \left|\cM_s(\omega)\right|^q \frac{ds}{s}
\right)^{1/q}.
\end{align*}

\begin{remark} \label{l1qfstarstar}
We comment in passing that the spaces $L^{(p,q)}$ discussed in \cite[Section 2.3]{LNT}, see also \cite[(4.39)]{Lions}, were defined as the space of functions $f$ such $s^{1/p}f^{\ast\ast}(s) \in L^q(ds/s)$. Note that, since $\cM_s(f)=sf^{\ast\ast}(s)$, this definition coincides with the one above for $p=1$.
\end{remark}

It is well-known that for $1\le q \le 2$, we have a continuous embedding $L^{(1,q)}(\T^2) \embeds H^{-1}(\T^2)$. This embedding is compact for $q < 2$, see for example \cite[Theorem 2.3]{LNT}. For $q=2$, the space $L^{(1,2)}(\T^2)$ is the largest rearrangement invariant space with a continuous (but not compact) embedding in $H^{-1}(\T^2)$ \cite{Lions}. The following proposition, due to P.L. Lions, provides sufficient conditions for a family of functions in $L^{(1,2)}$ to be precompact in $H^{-1}$, see \cite[Lemma 4.1]{Lions}:
\begin{proposition}[P.L. Lions]
\label{prop:lions}
  A family $\{ \omega^\nu \}_{\nu} \subset L^{(1,2)}(\T^2)$ is precompact in $H^{-1}(\T^2)$, if the following conditions hold:
  \begin{itemize}
  \item[(i)] There exists $C>0$, such that $\Vert \omega^\nu \Vert_{L^{(1,2)}(\T^2)} \le C$ uniformly in $\nu$,
  \item[(ii)] we have uniform decay
    \[
    \lim_{\delta \to 0} \left\{ \sup_{\nu} \int_0^\delta | \cM_s(\omega^\nu) |^2 \frac{ds}{s} \right\} = 0
    \]
  \end{itemize}
\end{proposition}

We recall that $C([0,T];w\text{-}X)$ denotes the space of continuous functions in time with values in $X$ endowed with the topology of weak convergence. We also recall the following lemma from \cite{LNT}:

\begin{lemma}[{\cite[Lemma 2.3]{LNT}}]
\label{lem:cpct}
Let $X$ be a reflexive, separable Banach space. Let $\{f_n\}$ be a bounded sequence in $C([0,T];X)$. Then $\{f_n\}$ is precompact in $C([0,T];X)$ if and only if the following two conditions hold:
\begin{itemize}
\item[(a)] $\{f_n\}$ is precompact in $C([0,T];w\text{-}X)$,
\item[(b)] For any $t\in [0,T]$ and for any sequence $t_n \to t$, we have that $\{f_n(t_n)\}$ is precompact in $X$.
\end{itemize}
\end{lemma}

As a consequence of Lemma \ref{lem:cpct} and Proposition \ref{prop:lions} we obtain:
\begin{corollary}
\label{cor:cpct}
Let $\{\omega^\nu\}_\nu$ be a family of functions in $C([0,T];H^{-1}(\T^2))$ and suppose that $\{\partial_t \omega^\nu \}_\nu$ is bounded in  $L^2((0,T);H^{-M}(\T^2))$ for some $M>1$. Then $\{\omega^\nu\}_\nu$ is precompact in $C([0,T];H^{-1}(\T^2))$, if the following conditions hold:
\begin{enumerate}
\item we have
\[\sup_{\nu} \sup_{t\in [0,T]} \Vert \omega^\nu(t) \Vert_{L^{(1,2)}(\T^2)} < \infty,\]
\item we have the following uniform decay in $\nu$ and $t$:
  \[
  \lim_{\delta \to 0} \left\{ \sup_{\nu} \sup_{t\in [0,T]} \int_0^\delta |\cM_s(\omega^\nu(t))|^2 \frac{ds}{s} \right\} = 0.
  \]
\end{enumerate}
\end{corollary}

\begin{proof}
We will check that conditions (a) and (b) in Lemma \ref{lem:cpct} hold true with $X=H^{-1}(\T^2)$.

From the boundedness in $L^\infty_t L^{(1,2)}_x$, it follows that $\{\omega^\nu\}$ is bounded in $L^\infty_tH^{-1}_x$ because $L^{(1,2)}_x \embeds H^{-1}_x$. Since we assumed that $\{\partial_t\omega^\nu\}$ bounded in $L^2((0,T);H^{-M}(\T^2))$ for some, possibly large, $M>1$, it follows that $\{\omega^\nu\}$ is equicontinuous from $[0,T]$ into $H^{-M}_x$. We may now use \cite[Lemma C.1]{Lions} to verify that (a) holds.

To check condition (b) let $\{t_\nu\}$ be a convergent sequence in $[0,T]$ and consider $\{\omega^\nu (\cdot,t_\nu)\}$. Then, hypotheses (1) and (2) imply that conditions (i) and (ii) of Proposition \ref{prop:lions} hold true, which in turn implies that $\{\omega^\nu(\cdot,t_\nu)\}$ is precompact in $H^{-1}_x$. This verifies condition (b) of Lemma \ref{lem:cpct}.

The proof is complete.
 \end{proof}

%The proof of Corollary \ref{cor:cpct} is almost identical to the proof of \cite[Thm. 2.4]{LNT}. We do not repeat the details here. We can now prove the following result:

\begin{theorem}
  \label{thm:L12}
  Let $u$ be a physically realizable solution of the Euler equations \eqref{eq:euler}, with forcing $f\in L^2((0,T);L^2(\T^2))$ and divergence-free initial data $u_0 \in L^2(\T^2)$. Let $u^\nu$ be a physical realization of $u$ and assume that the forcing $f^\nu \rightharpoonup f$ in $L^2((0,T);L^2(\T^2))$. Assume that $\omega_0^\nu = \curl(u^\nu_0)\in L^{(1,2)}(\T^2)$ and $g^\nu = \curl(f^\nu) \in L^1([0,T];L^{(1,2)}(\T^2))$, for all $\nu$. If we have
\begin{itemize}
\item uniform bounds
  \begin{align}
    \label{eq:L12-1}
    \sup_{\nu} \Vert \omega_0^\nu \Vert_{L^{(1,2)}}\le C, \;  \sup_{\nu} \Vert g^\nu \Vert_{L^1_tL^{(1,2)}_x} \le C,
  \end{align}
\item uniform decay of the vorticity initial data
  \begin{align}
    \label{eq:L12-2}
    \lim_{\delta \to 0}
    \left\{
    \sup_{\nu}
    \int_0^\delta |\cM_s(\omega_0^\nu)|^2\frac{ds}{s}
    \right\}
    =
    0,
  \end{align}
\item and uniform decay of the time-averaged forcing
  \begin{align}
    \label{eq:L12-3}
    \lim_{\delta \to 0}
    \left\{\sup_\nu
    \int_0^\delta \left|\int_0^T \cM_s(g^\nu(\tau)) \, d\tau \right|^2 \frac{ds}{s}
    \right\}
    =
    0,
  \end{align}
\end{itemize}
then the family $\{u^\nu\}_{\nu}$ is relatively compact in $C([0,T];L^2(\T^2))$, and $u$ is an energy balanced solution.
\end{theorem}

\begin{proof}
From the definition of a physical realization, see Definition \ref{def:physrealwsoln}, we already have $u^\nu \rightharpoonup u$ in weak-$\ast$ $L^\infty_tL^2_x$.

We will show that, under the assumptions of Theorem \ref{thm:L12}, the vorticities
$\{\omega^\nu = \curl (u^\nu) \}_\nu\}$ are relatively compact in $C([0,T];H^{-1}(\T^2))$. This in turn implies that $\{u^\nu\}_{\nu}$ is relatively compact in $C([0,T];L^2(\T^2))$ and hence, by Proposition \ref{prop:conv-->ebal}, that $u$ is energy balanced.

  We will use Corollary \ref{cor:cpct} to show that $\{\omega^\nu \}_\nu\}$ is relatively compact in $C([0,T];H^{-1}(\T^2))$. To this end we begin by observing that, for each fixed $\nu>0$, $\omega^\nu \in C([0,T];H^{-1}(\T^2))$. Indeed, this is an immediate consequence of the fact that $u^\nu \in C([0,T];L^2(\T^2))$, as noted in Remark \ref{viscencont}. Next, in the proof of Lemma \ref{lem:tcont} we deduced an estimate on $\partial_t u^\nu$, \eqref{visctderest}, from which it follows that $\{\partial_t \omega^\nu\}$ is bounded in $L^2_t H^{-M}_x$ for some, possibly large, $M>1$.

 It remains to show that our assumptions imply uniform control on $\Vert \omega^\nu(t) \Vert_{L^{(1,2)}} $ and that $ \displaystyle{\sup_\nu \sup_{t\in[0,T]} \int_0^\delta | \cM_s(\omega^\nu(t)) |^2 \frac{ds}{s}} \to 0$ as $\delta \to 0$. By Proposition \ref{prop:rearrange}, we have
  \begin{align*}
\cM_s (\omega^{\nu}(t))
&\le
\cM_s (\omega^{\nu}_0)
+
\int_0^t \cM_s (g^{\nu}(\tau)) \, d\tau.
\end{align*}
We can thus bound
\begin{align*}
\Vert \omega^{\nu,\epsilon}(t)\Vert_{L^{(1,2)}}
&=
\left(
\int_{0}^{|\T^2|}
|\cM_s (\omega^{\nu,\epsilon}(t))|^2
\, \frac{ds}{s}
\right)^{1/2}
\\
&\le
\left(
\int_{0}^{|\T^2|}
|\cM_s (\omega_0^{\nu})|^2
\, \frac{ds}{s}
\right)^{1/2}
\\
&\qquad
+
\left(
\int_{0}^{|\T^2|}
\left|
\int_0^T
\cM_s (g^{\nu}(\tau))
\, d\tau
\right|^2 \, \frac{ds}{s}
\right)^{1/2}
\end{align*}
Minkowski's integral inequality applied to $G(\tau,s) := \cM_s (g^\nu(\tau))$ implies that
\begin{align*}
\left[
\int_{0}^{|\T^2|}
\left[
\int_0^T
G(\tau,s)
\, d\tau
\right]^2 \, \frac{ds}{s}
\right]^{1/2}
\le
\int_0^T \left[
\int_0^{|\T^2|} G(\tau,s)^2 \, \frac{ds}{s}
\right]^{1/2} \, d\tau,
\end{align*}
and hence
\[
\Vert \omega^{\nu,\epsilon}(t) \Vert_{L^{(1,2)}}
\le
\Vert \omega_0^\nu \Vert_{L^{(1,2)}}
+
\int_0^t \Vert g^\nu(\tau) \Vert_{L^{(1,2)}} \, d\tau.
\]
By our assumptions on $\omega_0^\nu$ and $g^\nu$, the right-hand side is bounded uniformly in $\nu$ and $t\in [0,T]$. This is the first condition of Corollary \ref{cor:cpct}.

Next, replacing the upper integration bound $|\T^2|$ by $\delta > 0$, the same argument implies that
\begin{align*}
\left(\int_0^\delta |\cM_s (\omega^{\nu,\epsilon}(t))|^2 \frac{ds}{s} \right)^{1/2}
&\le
\left(\int_0^\delta |\cM_s(\omega^\nu_0)|^2 \frac{ds}{s}\right)^{1/2}
\\
&\qquad +
\left(\int_0^\delta \left| \int_0^T \cM_s(g^\nu(\tau))\, d\tau \right|^2 \frac{ds}{s} \right)^{1/2}.
\end{align*}
The right-hand side is independent of $t$. Furthermore, by our assumptions, the right-hand side converges to zero as $\delta \to 0$, uniformly in $\nu$. Thus, the family $\{\omega^{\nu}\}_{\nu}$ satisfies the assumptions of Corollary \ref{cor:cpct}, and hence $\{\omega_t^{\nu}\}_{\epsilon, \nu > 0}$ is precompact in $C([0,T];H^{-1}(\T^2))$.

This concludes the proof.
\end{proof}

In particular, the last theorem can be invoked if $\omega^\nu_0$ and $g^\nu$ satisfy uniform bounds in $L^{(1,q)}$ for $1\le q<2$, as shown next.

\begin{corollary} \label{cor:L1q}
   Let $u$ be a physically realizable solution of the Euler equations \eqref{eq:euler}, with forcing $f\in L^2((0,T);L^2(\T^2))$ and divergence-free initial data $u_0 \in L^2(\T^2)$. Let $u^\nu$ be a physical realization of $u$ and assume that the forcing $f^\nu \rightharpoonup f$ in $L^2((0,T);L^2(\T^2))$. Fix $1\le q <2$, and assume that $\omega_0^\nu = \curl(u^\nu_0)\in L^{(1,q)}(\T^2)$ and $g^\nu = \curl(f^\nu) \in L^1([0,T];L^{(1,q)}(\T^2))$ for all $\nu$. If we have uniform bounds
  \[
  \sup_{\nu} \Vert \omega_0^\nu \Vert_{L^{(1,q)}_x}\le C, \;  \sup_{\nu} \Vert g^\nu \Vert_{L^1_tL^{(1,q)}_x} \le C,
  \]
  then $\{u^\nu\}_{\nu}$ is relatively compact in $C([0,T];L^2(\T^2))$, and $u$ is an energy balanced solution.
\end{corollary}

\begin{proof}

  We note that the {\it a priori} $L^{(1,2)}$-bounds \eqref{eq:L12-1} on $\omega^\nu_0$ and $g^\nu$ follow immediately from the assumed $L^{(1,q)}$-bounds. In the following, we will show that the {\it a priori} $L^{(1,q)}$-bounds on $\omega_0^\nu$ and $g^\nu$ for $q<2$ imply the uniform decay conditions \eqref{eq:L12-2} and \eqref{eq:L12-3} of Theorem \ref{thm:L12}: Indeed, since $s\mapsto \cM_s(\omega^\nu_0)$ is a monotonically increasing function, we have
  \[
  \int_0^\delta \left| \cM_s(\omega^\nu_0) \right|^2 \, \frac{ds}{s}
  \le
  \left| \cM_\delta(\omega^\nu_0)\right|^{2-q} \, \int_0^\delta \left| \cM_s(\omega^\nu_0)\right|^q \, \frac{ds}{s}
  \le
  \left|\cM_\delta(\omega^\nu_0)\right|^{2-q} \Vert \omega^\nu_0 \Vert^{q}_{L^{(1,q)}},
  \]
  and
  \[
  \left|\cM_\delta(\omega^\nu_0)\right|^q \, \log\left(\frac{|\T^2|}{\delta}\right) \le \int_\delta^{|\T^2|} \left|\cM_s(\omega^\nu_0)\right|^q \frac{ds}{s} \le \Vert \omega^\nu_0 \Vert_{L^{(1,q)}}^q.
  \]
  Combining these estimates yields
  \[
  \int_0^\delta \left|\cM_s(\omega^\nu_0)\right|^2 \, \frac{ds}{s}
  \le
  \frac{\Vert \omega^\nu_0 \Vert_{L^{(1,q)}}^2}{\left| \,\log(|\T^2|/\delta)\right|^{(2-q)/q}}.
  \]
  Given that $\sup_\nu \Vert \omega^\nu_0 \Vert_{L^{(1,q)}} \le C$, the last upper bound decays to zero uniformly in $\nu$, as $\delta \to 0$. This shows the uniform decay condition \eqref{eq:L12-2}.

  Similarly, replacing $\cM_s(\omega^\nu_0)$ by $\int_0^T \cM_s(g^\nu(\tau)) \, d\tau$, we can show that
  \begin{align*}
  \int_0^\delta \left| \int_0^T \cM_s(g^\nu(\tau)) \, d\tau\right|^{2}\, \frac{ds}{s}
  &\le
  \frac{\left[\int_{0}^{|\T^2|} \left| \int_0^T \cM_s(g^\nu(\tau)) \, d\tau \right|^q \, \frac{ds}{s}\right]^{2/q}}{\left|\,\log(|\T^2|/\delta)\right|^{(2-q)/q}}
  \\
  &\le
  \frac{\left[\int_0^T \left( \int_{0}^{|\T^2|} \left|\cM_s(g^\nu(\tau))\right|^q \,
  \frac{ds}{s} \right)^{1/q} \, d\tau\right]^2 }{\left|\,\log(|\T^2|/\delta)\right|^{(2-q)/q}}
  \\
  &=
  \frac{\left[\int_0^T \Vert g^\nu(\tau) \Vert_{L^{(1,q)}} \, d\tau\right]^2 }{\log(|\T^2|/\delta)^{(2-q)/q}},
  \end{align*}
  where we have used Minkowski's integral inequality to pass to the second line. The boundedness assumption on $g^\nu$ now implies the uniform decay condition \eqref{eq:L12-3}. The claim thus follows from Theorem \ref{thm:L12}.
\end{proof}

We close this subsection with a result for the Orlicz spaces $L(\log L)^\alpha$, with $\alpha > 2$. We briefly recall that these spaces are defined as
\[L(\log L)^\alpha (\T^2) = \{ f \in L^1(\T^2) \, \big| \, \int_{\T^2} |f| [\log^+(|f|)]^\alpha \, dx < \infty\}.
\]
These are Banach spaces under the Luxembourg norm, given by
\[\|u\|_{L(\log L)^\alpha}:=\inf \left\{\lambda \, \bigg| \, \int_{\T^2} \left(\frac{|u(x)|}{\lambda}\right) \left[\log^+ \left(\frac{|u(x)|}{\lambda}\right)\right]^\alpha \,dx\leq 1\right\},\]
 see e.g. \cite{BennettSharpley}.

\begin{corollary} \label{cor:LlogLalpha}
   Let $u$ be a physically realizable solution of the Euler equations \eqref{eq:euler}, with forcing $f\in L^2((0,T);L^2(\T^2))$ and divergence-free initial data $u_0 \in L^2(\T^2)$. Let $u^\nu$ be a physical realization of $u$ and assume that the forcing $f^\nu \rightharpoonup f$ in $L^2((0,T);L^2(\T^2))$. Fix $\alpha > \displaystyle{\frac{1}{2}}$, and assume that $\omega_0^\nu = \curl(u^\nu_0)\in L(\log L)^\alpha(\T^2)$ and $g^\nu = \curl(f^\nu) \in L^1([0,T];L(\log L)^\alpha(\T^2))$ for all $\nu$. If we have uniform bounds
  \[
  \sup_{\nu} \Vert \omega_0^\nu \Vert_{L(\log L)^\alpha_x}\le C, \;  \sup_{\nu} \Vert g^\nu \Vert_{L^1_tL(\log L)^\alpha_x} \le C,
  \]
  then $\{u^\nu\}_{\nu}$ is relatively compact in $C([0,T];L^2(\T^2))$, and $u$ is an energy balanced solution.
\end{corollary}

\begin{proof}
The result is an immediate consequence of Corollary \ref{cor:L1q} and the embedding
\[L(\log L)^\alpha (\T^2) \subset L^{(1,1/\alpha)}(\T^2),\]
which was established in \cite[Lemma 2.1]{LNT}.
\end{proof}

\section{Conclusion}

The primary objective of this work is to find suitable conditions on the regularity of the forcing to characterize those physically realizable weak
solutions of the 2D Euler equations which are energy balanced, combining previous work by Lanthaler, Mishra and Par\'{e}s-Pulido in \cite{LMPP2021a} with work by Lopes Filho and Nussenzveig Lopes in \cite{LN2022}. In \cite{LMPP2021a}, the authors establish, for flows without forcing, equivalence between strong convergence of the viscous approximation in $L^2_t L^2_x$ and conservation of energy. We have proved the corresponding statement for flows with forcing, assuming $\{f^{\nu}\}$ converges strongly in $L^2_t L^2_x$. In addition, we prove that these equivalent conditions are also equivalent to strong convergence of the viscous velocities in $C([0,T];L^2(\T^2))$. Furthermore, even if we consider only initial vorticities in $L^p$, $1<p<2$, and the direction ``strong convergence $\Longrightarrow$ energy balance", our conditions on the forcing are weaker than those in \cite{LN2022}. Additionally, we provide examples of flows with vorticity in rearrangement-invariant spaces, for which the sufficient conditions are shown to hold. To this end, we develop novel {\it a priori} estimates for the rearrangement-invariant maximal vorticity function of solutions of the incompressible Navier-Stokes equations, and under minimal regularity assumptions. Our result fills a gap in the proof of \cite[Corollary 2.13]{LMPP2021a}, where such bounds were asserted without proof. In short, we have proven energy conservation of physically realizable solutions with initial vorticity belonging to arbitrary rearrangement-invariant spaces with compact embedding in $H^{-1}$, and under natural assumptions on the external force. In particular, this sharpens and extends the results of \cite{LN2022} in the forced setting, going beyond $L^p$ vorticity control.

We describe  a few avenues for future work. Firstly, the present work only considers deterministic forcing. Since investigations of driven turbulence often employ stochastic forcing, it would be of interest to extend the present work to the stochastic setting. Second, it would be interesting to consider the effects of a boundary, and investigate potential connections of the characterization of energy conservation in the present work with the Kato criterion \cite{kato1984seminar}. Thirdly, combining the ideas of the present work with those of \cite{ciampa2022energy} could provide a proof of energy conservation in the two-dimensional plane, going beyond vorticity in $L \log(L)^\alpha$ as assumed in \cite{ciampa2022energy}. We plan to explore some of these research directions in the future.

\section*{Acknowledgments}
The authors would like to thank Siddhartha Mishra for several insightful discussions which contributed significantly to this work. MCLF and HJNL are thankful for the generous hospitality of the ``Forschungsinstitut f\"ur Mathematik'' (FIM) at ETH Z\"urich during a research stay in Summer 2022.
The research of SL is supported by Postdoc.Mobility grant P500PT-206737 from the Swiss National Science Foundation.
MCLF was partially supported by CNPq, through grant \# 304990/2022-1, and FAPERJ, through  grant \# E-26/201.209/2021.
HJNL acknowledges the support of CNPq, through  grant \# 305309/2022-6, and of FAPERJ, through  grant \# E-26/201.027/2022.

\appendix
\section{Operator splitting for advection-diffusion}
\label{app:adv-diff}

In this appendix, our goal is to provide a self-contained proof of convergence for an operator splitting approximation of a simple advection-diffusion equation. We fix a smooth divergence-free vector field $U \in C^\infty(\T^2\times [0,T])$, a scalar forcing $g\in C^\infty(\T^2\times[0,T])$ with zero mean $\int_{\T^2} g(x,t)\, dx = 0$ for all $t\in [0,T]$, and we consider the following advection-diffusion PDE:
\begin{gather}
  \label{eq:adv-diff}
  \left\{
  \begin{array}{ll}
    \partial_t \beta + U\cdot \nabla \beta = \nu \Delta \beta + g, \quad & \text{in } \T^2 \times (0,T), \\
    \beta(t=0) = \beta_0, &\text{on } \T^2 \times \{0\}.
  \end{array}
  \right.
\end{gather}
We will assume throughout that the initial data $\beta_0 \in C^\infty(\T^2)$, and $\beta_0$ has zero mean, so that $\int_{\T^2} \beta_0(x) \, dx = 0$.

\begin{remark}
While several convergence results for operator splitting approximations of advection-diffusion-reaction equations are available in the literature, those results mostly focus on higher-order (e.g. Strang-) operator splitting procedures. We have not been able to find a simple reference for the exact PDE \eqref{eq:adv-diff} and the low-order splitting of relevance for the present work. Our goal in this appendix is therefore to provide a self-contained proof of convergence of a low-order operator splitting scheme for \eqref{eq:adv-diff}. From a numerical analysis point of view, the estimates in this appendix mostly follow standard procedure. We claim no originality, but include them here for completeness.
\end{remark}

\subsection{Operator splitting}

\subsubsection{Heat equation}
Let $t,t_0 \in [0,T]$ with $t\ge t_0$. In the following, we will denote by $\beta_0 \mapsto H(t;t_0) \beta_0$ the solution operator of the following forced heat equation:
\begin{align}
  \label{eq:H}
  \left\{
  \begin{aligned}
    &\partial_t \beta(t) = \nu \Delta \beta(t) + g(t), \\
    &\beta(\slot, t_0) = \beta_0,
  \end{aligned}
  \right.
\end{align}
i.e. $t \mapsto H(t;t_0)\beta_0$ solves \eqref{eq:H} over the time interval $[t_0,T]$.
We recall that $g(x,t)$ depends explicitly on $t$, and therefore the solution operator $H(t;t_0)$ depends on both the initial time $t_0$, as well as on $t$.

\subsubsection{Transport equation}

Similarly, we denote by $\beta_0 \mapsto E(t;t_0) \beta$ the solution operator of the transport equation:
\begin{align}
  \label{eq:E}
  \left\{
  \begin{aligned}
    &\partial_t \beta + U \cdot \nabla \beta = 0, \\
    &\beta(\slot, t) = \beta_0.
  \end{aligned}
  \right.
\end{align}
So that $t \mapsto E(t;t_0) \beta_0$ solves \eqref{eq:E}.

\subsubsection{Splitting scheme}

We expect that the solution operator $S(t;t_0)$ for the advection-diffusion equation \eqref{eq:adv-diff} can be approximated by
\[
S(t;t_0) \approx H(t;t_0)E(t;t_0) + O(\Delta t^2), \quad \text{for } |t-t_0| \le \Delta t,
\]
and hence, repeated application of $H$ and $E$ over small time-steps of size $\Delta t$ are expected to converge to the true solution operator as $\Delta t \to 0$. Our goal is to make this intuition precise, in the following. To this end, we will show that the corresponding ``operator splitting approximant'' $\beta^\Delta$ converges as $\Delta t \to 0$:

To be more precise, given initial data $\beta_0$ at $t=0$, a finite time-horizon $T>0$ and a small time-step $\Delta t = \displaystyle{\frac{T}{N}} > 0$, we set $t_n:= n\Delta t$ for $n=0,1,\ldots, N$, and we define a sequence ${\beta}^{\Delta}_n(t)$ recursively, by
\begin{align}
  \label{eq:opspl1}
{\beta}^{\Delta}_0(0) := \beta_0,
\end{align}
and
\begin{align}
  \label{eq:opspl2}
  \left\{
  \begin{aligned}
{\beta}^{\Delta}_{n+1/2}(t)
&:=
E(t;t_n) {\beta}^{\Delta}_n(t_n),
\\
{\beta}^{\Delta}_{n+1}(t)
&:=
H(t;t_n) {\beta}^{\Delta}_{n+1/2}(t),
\end{aligned}
\right.
\qquad
\text{for } t \in (t_n,t_{n+1}].
\end{align}
We furthermore define ${\beta}^{\Delta}$ for any $t\in [0,T]$ by
\begin{align}
\label{eq:opspl3}
{\beta}^{\Delta}(t)
:=
{\beta}^{\Delta}_{n}(t),
\quad \text{for } t\in (t_{n-1},t_n].
\end{align}

\subsection{Convergence}
Let $t\mapsto \beta(t)$ denote the exact solution to \eqref{eq:adv-diff} with initial data $\beta_0 \in C^\infty(\T^2\times[0,T])$, satisfying $\int_{\T^2}\beta_0(x)\, dx =0$. Our first goal is to show that the operator splitting approximation converges ${\beta}^{\Delta} \to \beta$ as $\Delta t \to 0$.

To this end, we will need to derive several basic stability estimates. We start with the following lemma:

\begin{lemma}
\label{lem:Hstab}
Let $\sigma \ge 0$. Let $t,t_0\in [0,T]$ with $t\ge t_0$. Assume smooth forcing $g\in C^\infty(\T^2\times [0,T])$ in the forced heat equation \eqref{eq:H}, and $\int_{\T^2} g(x)\, dx = 0$. If $|t-t_0|\le \Delta t$, then
\begin{align}
\label{eq:Hstab}
\Vert H(t;t_0){\beta_0} \Vert_{H^\sigma_x}
\le
\Vert \beta_0 \Vert_{H^\sigma_x}
+
\Delta t \Vert g \Vert_{L^\infty_t H^\sigma_x}.
\end{align}

Furthermore, if $k  \in \mathbb{N}$ and $\Delta$ is the Laplacian then
\begin{align}
\label{eq:HstabL2}
\Vert \Delta^k H(t;t_0){\beta_0} \Vert_{L^2_x}
\le
\Vert \Delta^k \beta_0 \Vert_{L^2_x}
+
\Delta t \Vert \Delta^k g \Vert_{L^\infty_t L^2_x}.
\end{align}
\end{lemma}

\begin{proof}
Both \eqref{eq:Hstab} and \eqref{eq:HstabL2} follow in a straightforward manner from the representation of the solution in terms of the heat kernel $G_t$:
\[
H(t;t_0) \beta_0
=
G_{t-t_0} \ast \beta_0
+
\int_{t_0}^{t} G_{t-\tau} \ast g(\slot, \tau) \, d\tau.
\]

\end{proof}

\begin{lemma}[Stability estimate for ${\beta}^\Delta$]
\label{lem:stabtilde}
Assume that $\beta_0$, $g$ and $U$ are smooth functions, $\dv(U) = 0$ and $\beta_0$, $g$ have zero mean. Let $\beta^\Delta$ be defined by \eqref{eq:opspl3}.
For any $\sigma >0$, there exists a constant $C = C(T,\sigma,g,U,\beta_0)>0$, such that
\[
\sup_{t\in [0,T]} \Vert {\beta}^\Delta(t) \Vert_{H^\sigma_x}
\le
C,
\]
and
\[
\max_{n:\, t_n < T}\sup_{t\in [t_n,t_{n+1}]} \Vert E(t;t_n) \beta^\Delta_n(t_n) \Vert_{H^\sigma_x} \le C,
\]
uniformly as $\Delta t\to 0$.
\end{lemma}

\begin{proof}
Recall that $\beta^\Delta(t)$ is defined piecewise as $\beta^\Delta(t) = \beta^\Delta_n(t)$ over intervals $t \in (t_{n-1},t_{n}]$ with $t_n = n\Delta t$ for $n = 1,2,\dots$. By definition, we have the recursion
\[
\beta^\Delta_n(t) = H(t;t_n)E(t;t_n) \beta^\Delta_{n-1}(t_n).
\]
To provide a uniform bound on $\Vert \beta^\Delta(t) \Vert_{H^\sigma_x}$ for all $t\in [0,T]$, we will proceed in three steps: In a first step, we derive a general $H^\sigma_x$-estimate over short time-intervals of length $\Delta t$. Then, we use this estimate to bound the values $\Vert \beta^\Delta(t_n) \Vert_{H^\sigma_x}$ at the interval endpoints. Finally, we combine the short-time estimate of Step 1 and the uniform bound on the $\Vert \beta^\Delta(t_n) \Vert_{H^\sigma_x}$ from Step 2, to conclude that there exists a uniform bound on $\Vert \beta^\Delta(t) \Vert_{H^\sigma_x}$ for all $t\in [0,T]$.

\textbf{Step 1: (short-time estimate)} We begin by claiming that, for any $\sigma > 0$, $\|H(t;t_0) E(t;t_0) \beta_0 \|_{H^\sigma_x}$ is bounded on short time intervals $|t-t_0| < \Delta t$ in terms of $\| \beta_0 \|_{H^{\sigma}}$ and $\Delta t$. By interpolation, it suffices to prove the claim for $\sigma = 2k$, where $k\in \N$. We first recall that for any $t_0\le t$, and $\beta_0\in C^\infty$ with zero-mean, we have (cf. Lemma \ref{lem:Hstab}), \eqref{eq:Hstab}:
\begin{align}
\label{eq:HE1}
\Vert H(t;t_0) E(t;t_0) \beta_0 \Vert_{H^\sigma_x}
\le
\Vert E(t;t_0) \beta_0 \Vert_{H^\sigma_x}
+
\Delta t \Vert g \Vert_{L^\infty_tH^\sigma_x}.
\end{align}
Let us denote $\tilde{\beta}(t) := E(t;t_0)\beta_0$. To estimate $\Vert \tilde{\beta} \Vert_{H^\sigma_x} = \Vert E(t;t_0) \beta_0 \Vert_{H^\sigma_x}$ for $\sigma = 2k$, we multiply \eqref{eq:E} by $\Delta^{2k}\tilde{\beta}$, with $\Delta$ the Laplacian, to find
\begin{align*}
\frac{d}{dt} \frac12 \Vert \Delta^k \tilde{\beta} \Vert_{L^2_x}^2
&\le
\left|
\langle U \cdot \nabla \tilde{\beta}, \Delta^{2k} \tilde{\beta} \rangle_{L^2_x}
\right|
\\
&=
\left|
\langle \Delta^{k} \left( U \cdot \nabla \tilde{\beta}\right), \Delta^{k} \tilde{\beta} \rangle_{L^2_x}
\right|
\\
&\le
\left|
\langle U \cdot \nabla (\Delta^{k}\tilde{\beta}), \Delta^{k} \tilde{\beta} \rangle_{L^2_x}
\right|
+
C_k \Vert U \Vert_{W^{2k,\infty}_x} \Vert \tilde{\beta} \Vert_{H^{2k}_x}^2,
\end{align*}
with a constant $C_k$ depending only on $k$. Since $U$ is divergence-free, the first term vanishes on account of the periodic boundary conditions,
\[
\left|
\langle U \cdot \nabla (\Delta^{k}\tilde{\beta}), \Delta^{k} \tilde{\beta} \rangle_{L^2_x}
\right|
=
\left|
\int_{\T^2} \dv\left(\frac12 \left[\Delta^{k}\tilde{\beta}\right]^2 U \right)
\right|
= 0.
\]
Since $\beta_0$, $g$ are assumed to have zero mean, it follows that also $\int_{\T^2} \tilde{\beta}(x,t) \, dx = 0$ at later times, and hence we may use the Poincar\'e inequality to obtain equivalence of norms:
\[
\Vert \Delta^k \tilde{\beta} \Vert_{L^2_x}
\le
\Vert \tilde{\beta} \Vert_{H^{2k}_x}
\le
C_k
\Vert \Delta^k \tilde{\beta} \Vert_{L^2_x}.
\]
Gronwall's lemma applied to the differential inequality
\[
\frac{d}{dt} \Vert \Delta^k \tilde{\beta} \Vert^2_{L^2_x}
\lesssim_{k,U} \Vert \Delta^k \tilde{\beta} \Vert^2_{L^2_x},
\]
implies that $\Vert \Delta^k \tilde{\beta}(t) \Vert^2_{L^2_x}
\leq e^{C|t-t_0|} \Vert \Delta^k \beta_0 \Vert_{L^2_x}^2 $, with $C$ depending only on $k$ and $U$. Therefore,
\begin{equation} \label{EstabL2}
\Vert \Delta^k \tilde{\beta}(t) \Vert^2_{L^2}\leq (1+C|t-t_0|) \Vert \Delta^k \beta_0 \Vert^2_{L^2},
\end{equation}
where the implied constant in the second estimate is uniform for $|t-t_0| \le T$. By the equivalence of norms and upon rewriting $\sigma =2k$, we conclude that there exists a constant $C = C(\sigma,U) >0$, such that
\begin{align}
\label{eq:HE2}
\Vert E(t;t_0) \beta_0 \Vert_{H^\sigma_x}
\le
C(1+\Delta t)^{1/2} \Vert \beta_0 \Vert_{H^{\sigma}},
\end{align}
whenever $|t-t_0|\le \Delta t$. Combining \eqref{eq:HE1} and \eqref{eq:HE2}, we have shown that there exists a constant $C = C(\sigma,U,g)>0$, such that
\begin{align}
\label{eq:HE3}
\Vert H(t;t_0) E(t;t_0) \beta_0 \Vert_{H^\sigma_x}
\le
C(1+\Delta t)^{1/2}  \Vert \beta_0 \Vert_{H^{\sigma}}
+
C \Delta t,
\end{align}
if $|t-t_0|\le \Delta t \le T$.

Furthermore, from \eqref{eq:HstabL2} and \eqref{EstabL2}, we also have
\begin{align}
\label{eq:HE3new}
\Vert \Delta^k H(t;t_0) E(t;t_0) \beta_0 \Vert_{L^2_x}
\le
(1+C\Delta t)^{1/2}  \Vert \Delta^k\beta_0 \Vert_{L^2_x}
+
\Vert \Delta^k g \Vert_{L^\infty_t L^2_x} \Delta t,
\end{align}
if $|t-t_0|\le \Delta t \le T$.

\textbf{Step 2: (estimate for $t=t_n$) }Recall that ${\beta}^\Delta$ is defined recursively by application of $H(t;t_n)E(t;t_n)$ over short time-intervals of length $\Delta t$. Using \eqref{eq:HE3new}, we thus arrive at the recursive estimate
\[
\Vert \Delta^k{\beta}_{n+1}^\Delta(t_{n+1})\Vert_{L^2_x}
\le
(1+C\Delta t)^{1/2} \Vert \Delta^k {\beta}_{n}^\Delta(t_n) \Vert_{L^2_x}
+ \Vert \Delta^k g \Vert_{L^\infty_t L^2_x} \Delta t.
\]
Iterating this inequality backwards until $n=0$ yields
\begin{align} \label{eq:iterestL2}
\Vert \Delta^k{\beta}_{n+1}^\Delta(t_{n+1})\Vert_{L^2_x}
& \le
(1+C\Delta t)^{(n+1)/2} \Vert \Delta^k {\beta}_{0}^\Delta(0) \Vert_{L^2_x}\\
& + \Vert \Delta^k g \Vert_{L^\infty_t L^2_x} \Delta t \sum_{j=0}^n (1+C\Delta t)^{j/2}.
\end{align}

Recall $t_n=n\Delta t$, with $n=0, 1,\ldots,N$. Then, the first term on the right-hand-side of \eqref{eq:iterestL2} is bounded by
\[\left(1+C\frac{T}{N}\right)^{N/2} \Vert \Delta^k {\beta}_{0}^\Delta(0) \Vert_{L^2_x}, \]
which, in turn, converges to $e^{CT/2}\Vert \Delta^k {\beta}_{0}^\Delta(0) \Vert_{L^2_x}$ as $N\to \infty$.

The second term on the right-hand-side of \eqref{eq:iterestL2} is bounded by
\[\Vert \Delta^k g \Vert_{L^\infty_t L^2_x} \frac{T}{N} \sum_{j=0}^{N-1} \left(1+C\frac{T}{N}\right)^{j/2},\]
which converges, as $N\to\infty$, to
\[2\frac{\Vert \Delta^k g \Vert_{L^\infty_t L^2_x} }{C}(e^{CT/2}-1).\]

Therefore it follows that
\[\Vert \Delta^k{\beta}_{n}^\Delta(t_{n})\Vert_{L^2_x} \lesssim_{T,U,g,k}\Vert \Delta^k\beta_0\Vert_{L^2_x} + 1,\]
for all $n=0,1,\ldots,N$, with an implied constant depending only on $T, U, g,k$. Thus, from the equivalence of norms we have that there exists a constant $C>0$ depending only on $T$, $\sigma$, $U$, $g$ and on ${\beta}^\Delta_0(t_0) = \beta_0$, such that
\[
\max_{t_n \le T}
\Vert {\beta}^\Delta(t_n)\Vert_{H^\sigma_x}
\le
C.
\]

%The discrete form of Gronwall's inequality now readily implies that
%\[
%\Vert {\beta}^\Delta(t_n)\Vert_{H^\sigma_x}
%= \Vert \beta^\Delta_n(t_n)\Vert_{H^\sigma_x}
%\lesssim_{T,\sigma,U,g}
%\Vert \beta_0^\Delta(t_0) \Vert_{H^\sigma_x} + 1
%,
%\]
%is uniformly bounded for $t_n\in [0,T]$,

\textbf{Step 3: (conclusion)} Appealing once more to the inequality \eqref{eq:HE3}, i.e.
\begin{align*}
\Vert \beta^\Delta(t) \Vert_{H^\sigma_x}
&=
\Vert H(t;t_n) E(t;t_n) \beta^{\Delta}(t_n) \Vert_{H^\sigma_x}
\\
&\le
C(1+\Delta t)  \Vert \beta^\Delta(t_n) \Vert_{H^{\sigma}}
+
C \Delta t,
\end{align*}
it follows that there exists a constant $C = C(T,\sigma,U,g,\beta_0)>0$, such that
\[
\sup_{t\in [0,T]}
\Vert \beta^\Delta(t) \Vert_{H^\sigma_x} \le C.
\]
This is the claimed upper bound.
\end{proof}

The solution operator of the heat equation $H(t;t_0)$ evaluated at $t=t_0$ is identity, $H(t_0;t_0) = I$. Formally expanding in $t$, we expect that $H(t;t_0) = I + O(\Delta t)$.
The following lemma formalizes this fact (with a very crude estimate for the first-order correction term):

\begin{lemma}
\label{lem:Hid}
Let $\beta_0\in C^\infty(\T^2)$ be a smooth function, and let $t_0\in [0,T]$. Then for any $t\in [t_0, t_0+\Delta t]$, and $\sigma \ge 0$, we have
\begin{align}
\label{eq:Hid}
\begin{aligned}
\Vert H(t;t_0) \beta_0 - \beta_0 \Vert_{H^\sigma_x}
&\le
\Delta t \, \nu \Vert \beta_0 \Vert_{H_x^{2+\sigma}}
\\
&\qquad
+
\Delta t (1+\nu \Delta t) \Vert g \Vert_{L^\infty_t H_x^{2+\sigma}}.
\end{aligned}
\end{align}
If $G_\tau$ denotes the heat kernel, then we similarly have
\begin{align}
\label{eq:Gid}
\Vert G_{t-t_0} \ast \beta_0 - \beta_0 \Vert_{H^\sigma_x}
\le
\Delta t \, \nu \Vert \beta_0 \Vert_{H_x^{2+\sigma}}
\end{align}
\end{lemma}

\begin{proof}
Let $h(t) := H(t;t_0)\beta_0$. By definition of $H(t;t_0)$, we have
\[
\partial_{t} h(t) = \nu \Delta h(t) + g(t),
\]
and $h(t_0) = \beta_0$. Integration in time yields
\[
 H(t;t_0) \beta_0 - \beta_0
 =
  \nu \int_{t_0}^{t}\Delta h(\tau) \, d\tau
 +
 \int_{t_0}^{t} g(\tau) \, d\tau,
\]
and we can estimate
\begin{align*}
\Vert H(t;t_0) \beta_0 - \beta_0 \Vert_{H^\sigma_x}
&\le
 \nu \int_{t_0}^{t} \Vert \Delta h(\tau) \Vert_{H^\sigma_x} \, d\tau
 +
 \int_{t_0}^{t} \Vert  g(\tau) \Vert_{H^\sigma_x} \, d\tau
 \\
 &\le
 \Delta t\, \nu  \Vert h \Vert_{L^\infty_t H_x^{2+\sigma}}
 +
 \Delta t\, \Vert  g\Vert_{L^\infty_t H_x^\sigma}.
\end{align*}
Taking into account \eqref{eq:Hstab}, we have
\[
\Vert h(t) \Vert_{H^{2+\sigma}_x}
=
\Vert H(t;t_0) \beta_0 \Vert_{H^{2+\sigma}}
\le
\Vert \beta_0 \Vert_{H^{2+\sigma}} + \Delta t \Vert g \Vert_{L^\infty_t H^\sigma_x}.
\]
Upon substitution of this bound, we thus obtain the (rough) estimate
\[
\Vert H(t;t_0) \beta_0 - \beta_0 \Vert_{H^\sigma_x}
\le
\Delta t \nu \Vert \beta_0 \Vert_{H_x^{2+\sigma}}
+
\Delta t (1+\nu \Delta t) \Vert g \Vert_{L^\infty_t H_x^{2+\sigma}}.
\]
The estimate for $G_{t-t_0}\ast \beta_0$ is the special case where $g\equiv 0$.
\end{proof}

Using the last lemma, we next show that the operator splitting approximant ${\beta}^\Delta$ is an approximate solution of the relevant equation, up to an $O(\Delta t)$ error.

\begin{lemma}
\label{lem:tildeeq}
Assume that $\beta_0$, $U$ and $g$ are smooth, $\dv(U) = 0$ and that $\beta_0, g$ have zero mean. Then the function ${\beta}^\Delta$ defined by \eqref{eq:opspl3} is continuous, and with the potential exception of finitely many break-points $t=t_0,t_1,\dots$, the function $\beta^\Delta$ solves the following PDE:
\begin{align}
\label{eq:opspliteq}
\partial_t {\beta}^\Delta + U \cdot \nabla {\beta}^\Delta
=
\nu \Delta {\beta}^\Delta + g + F^\Delta,
\end{align}
where $F^\Delta$ can be estimated by
\begin{align}
\label{eq:Fdelta}
\Vert F^\Delta \Vert_{L^\infty_t L^2_x} \le C\Delta t,
\end{align}
with a constant $C = C(T,\nu, U,g, \beta_0) > 0$ which is bounded
uniformly in $\Delta t$.
\end{lemma}

\begin{proof}
Continuity of ${\beta}^\Delta$ is straight-forward. We therefore focus on \eqref{eq:opspliteq}. To this end, we consider $t \in [t_n,t_{n+1}]$ for $n\ge 0$. By definition, we have
\[
\beta^\Delta(t) =
{\beta}^\Delta_{n+1}(t)
=
 H(t;t_n) \left[ E(t;t_n) {\beta}_n^\Delta(t_n) \right],
 \quad \forall \, t\in [t_n,t_{n+1}].
\]
For $\tau_1,\tau_2 \in [t_n,t_{n+1}]$, we now define,
\[
h(\tau_1,\tau_2) := H(\tau_1;t_n)\left[ E(\tau_2;t_n) {\beta}_n^\Delta(t_n) \right],
\]
so that ${\beta}^\Delta_{n+1}(t) = h(t,t)$. Note that the dependency on the spatial variable has been suppressed in this notation, i.e. $h$ is considered as a mapping
\[
h: [t_n,t_{n+1}] \times [t_n, t_{n+1}]
\to C^\infty(\T^2).
\]
We next observe that
\begin{align}
\label{eq:htau}
\partial_t \beta^\Delta_{n+1}(t)
=
(\partial_{\tau_1} h)(t,t) + (\partial_{\tau_2} h)(t,t).
\end{align}
We will compute the partial derivatives with respect to $\tau_1, \tau_2$ on the right-hand side of \eqref{eq:htau}.

\textbf{Calculation for $\partial_{\tau_1}$:} With $\tau_2$ fixed, let $\bar{\beta}(\tau_2) := E(\tau_2;t_n) \beta_0$. By definition of $H(\tau_1;t_n)$, the function $\tau_1 \mapsto h(\tau_1,\tau_2) = H(\tau_1;t_n) \bar{\beta}(\tau_2)$ solves the heat equation with initial data $\bar{\beta}(\tau_2)$. In particular,
\begin{align}
\label{eq:htau1}
\partial_{\tau_1} h(\tau_1,\tau_2) = g(\tau_1) + \nu \Delta h(\tau_1,\tau_2).
\end{align}

\textbf{Calculation for $\partial_{\tau_2}$:} To compute the derivative with respect to $\tau_2$, we freeze $\tau_1$, and note that $H(\tau_1;t_n)$ is explicitly given by
\[
H(\tau_1;t_n) \bar{\beta}(\tau_2)
= G_{\tau_1 - t_n} \ast \bar{\beta}(\tau_2)
+ \int_{t_n}^{\tau_1} G_{\tau_1 - \tau} \ast g(\tau) \, d\tau,
\]
in terms of the heat kernel $G_\tau$. Since the second term is independent of $\tau_2$, upon taking a partial derivative of $h(\tau_1,\tau_2) = H(\tau_1;t_n) \bar{\beta}(\tau_2)$ with respect to $\tau_2$, we obtain,
\[
\partial_{\tau_2} h(\tau_1, \tau_2)
=
G_{\tau_1-t_n} \ast \partial_{\tau_2} \bar{\beta}(\tau_2).
\]
By definition, $\tau_2 \mapsto \bar{\beta}(\tau_2) = E(\tau_2;t_n) \beta_n^\Delta(t_n)$ solves the transport equation, and hence,
\[
\partial_{\tau_2} \bar{\beta}(\tau_2)
=
- U(\tau_2) \cdot \nabla \bar{\beta}(\tau_2).
\]
Substitution of this identity in our equation for $\partial_{\tau_2}h$ yields,
\begin{align}
\label{eq:htau2}
\partial_{\tau_2} h(\tau_1, \tau_2)
=
- G_{\tau_1-t_n} \ast \left( U(\tau_2) \cdot \nabla \bar{\beta}(\tau_2) \right).
\end{align}

\textbf{Deriving the equation for $\beta^\Delta_n$:}
Combining \eqref{eq:htau}, \eqref{eq:htau1} and \eqref{eq:htau2}, and recalling that $\beta^{\Delta}_n(t) = h(t,t)$, we obtain,
\[
\partial_t \beta^{\Delta}_n(t)
=
-G_{t-t_n} \ast \left( U(t) \cdot \nabla \bar{\beta}(t) \right)
+
g(t) + \nu \Delta \beta^{\Delta}_n(t).
\]
Setting
\[
F^\Delta(t) :=
U(t) \cdot \nabla \beta^{\Delta}_n(t)
-G_{t-t_n} \ast \left( U(t) \cdot \nabla \bar{\beta}(t) \right),
\]
we thus have
\[
\partial_t \beta^{\Delta}_n(t)
+ U(t) \cdot \nabla \beta^{\Delta}_n(t)
=
g(t) + \nu \Delta \beta^{\Delta}_n(t) + F^\Delta(t).
\]
This is \eqref{eq:opspliteq}.

To conclude the proof of Lemma \ref{lem:tildeeq}, it thus remains to derive the bound \eqref{eq:Fdelta} on $F^\Delta(t)$.

\textbf{Estimate for $F^\Delta(t)$:} We can bound,
\begin{align*}
\Vert F^\Delta(t) \Vert_{L^2_x}
&=
\Vert
U(t) \cdot \nabla \beta^{\Delta}_n(t)
-G_{t-t_n} \ast \left( U(t) \cdot \nabla \bar{\beta}(t) \right)
\Vert_{L^2_x}
\\
&=
\Vert
U(t) \cdot \nabla H(t;t_0)\bar{\beta}(t)
-G_{t-t_n} \ast \left( U(t) \cdot \nabla \bar{\beta}(t) \right)
\Vert_{L^2_x}
\\
&\le
\Vert
U(t) \cdot \nabla \left\{ H(t;t_0)\left[\bar{\beta}(t)\right] - \bar{\beta}(t)\right\}
\Vert_{L^2_x}
\\
&\qquad
+
\Vert
U(t) \cdot \nabla \bar{\beta}(t)
-G_{t-t_n} \ast \left( U(t) \cdot \nabla \bar{\beta}(t) \right)
\Vert_{L^2_x}
\\
&=: (I) + (II).
\end{align*}
Applying Lemma \ref{lem:Hid}, bound \eqref{eq:Hid}, to the first term, we obtain
\begin{align*}
(I) &=
\big\Vert
U(t) \cdot \nabla \left\{ H(t;t_n)\left[\bar{\beta}(t)\right] - \bar{\beta}(t)\right\}
\big\Vert_{L^2_x}
\\
&\le
\Vert U(t) \Vert_{L^\infty}
\left\Vert H(t;t_n)[\bar{\beta}(t)] - \bar{\beta}(t) \right\Vert_{H^1}
\\
&\le \Vert U(t) \Vert_{L^\infty}
\left\{
\Delta t\, \nu  \left\Vert \bar{\beta}(t) \right\Vert_{H^3} + \Delta t\, (1+\nu \Delta t)  \left\Vert g(t) \right\Vert_{H^3}
\right\}.
\end{align*}
By Lemma \ref{lem:stabtilde}, we can bound $\Vert \bar{\beta}(t)\Vert_{H^3}$ by a constant $C = C(T, g, U, \beta_0)>0$, uniform in time and in $\Delta t$, i.e.
\begin{align}
\label{eq:bbb}
\sup_{t \in [t_n,t_{n+1}]} \Vert \bar{\beta}(t) \Vert_{H^3}
\le C.
\end{align}
Enlarging the constant $C$, if necessary, we thus have
\[
(I)
\le C \Delta t,
\]
where $C = C(T,\nu, g,U,\beta_0)$ is independent of $t$ and $\Delta t$.

Finally, by Lemma \ref{lem:Hid}, bound \eqref{eq:Gid}, we have
\begin{align*}
(II)
&=
\Vert
U(t) \cdot \nabla \bar{\beta}(t)
-G_{t-t_n} \ast \left( U(t) \cdot \nabla \bar{\beta}(t) \right)
\Vert_{L^2_x}
\\
&\le \Delta t \, \nu \Vert U(t) \cdot \nabla \bar{\beta}(t) \Vert_{H^2}
\\
&\le \Delta t \, \nu \Vert U(t) \Vert_{W^{2,\infty}_x} \Vert \bar{\beta}(t) \Vert_{H^3}.
\end{align*}
Invoking \eqref{eq:bbb}, it follows that
\[
(II) \le C \Delta t,
\]
where $C = C(T,\nu, g,U,\beta_0)>0$.

We conclude that
\[
\Vert F^\Delta(t) \Vert_{L^2_x}
\le (I) + (II) \le C\Delta t,
\]
for a constant $C = C(T,\nu, g,U,\beta_0)>0$, independent of $t$ and $\Delta t$. The claimed bound on $F^\Delta(t)$ thus follows upon taking the supremum over all $t \in [t_n,t_{n+1}]$ and all $n\in \N$ such that $t_n \le T$. This concludes our proof of Lemma \ref{lem:tildeeq}.
\end{proof}

We can finally state the following convergence result for the operator splitting approximation:

\begin{proposition}[Convergence of operator splitting]
\label{prop:splitconv}
Let $\beta_0 \in C^\infty(\T^2)$ be initial data for the advection-diffusion equation \eqref{eq:adv-diff}, with forcing $g \in C^\infty(\T^2\times [0,T])$ and divergence-free advecting velocity field $U\in C^\infty(\T^2\times [0,T])$. Assume that $\int_{\T^2} \beta_0(x) \, dx = 0$ and $\int_{\T^2} g(x) \, dx = 0$. Let $\beta$ be the solution of the advection-diffusion PDE \eqref{eq:adv-diff}, and let ${\beta}^\Delta$ be given by the operator splitting approximation \eqref{eq:opspl1},\eqref{eq:opspl2}. Then
\[
\lim_{\Delta t\to 0} \Vert \beta - {\beta}^\Delta \Vert_{L^\infty_t L^2_x} = 0.
\]
\end{proposition}

\begin{proof}
We define $\tilde{r}^\Delta := {\beta}^\Delta - \beta$. Then, by Lemma \ref{lem:tildeeq}, $\tilde{r}^\Delta$ solves
\[
\partial_t \tilde{r}^\Delta
+
U \cdot \nabla \tilde{r}^\Delta
=
\nu \Delta \tilde{r}^\Delta
+
F^\Delta,
\]
where $\Vert F^\Delta \Vert_{L^2_x} \le C\Delta t$ for some constant $C$ independent of $\Delta t$. Multiplying by $\tilde{r}^\Delta$, integrating over space and employing the Poincar\'e inequality, $\Vert \tilde{r}^\Delta \Vert_{L^2_x} \le C\Vert \nabla \tilde{r}^\Delta \Vert_{L^2_x} $, readily yields
\begin{align*}
\frac{d}{dt} \Vert \tilde{r}^\Delta \Vert_{L^2_x}^2
&\le
-2\nu \Vert \nabla \tilde{r}^\Delta \Vert_{L^2_x}^2 + 2\Vert F^\Delta \Vert_{L^2_x} \Vert \tilde{r}^\Delta \Vert_{L^2_x}
\\
&\le
-2\nu \Vert \nabla \tilde{r}^\Delta \Vert_{L^2_x}^2 + C\nu \Vert \tilde{r}^\Delta \Vert_{L^2_x}^2 + (C\nu)^{-1} \Vert F^\Delta \Vert_{L^2_x}^2
\\
&\le
(C\nu)^{-1} \Vert F^\Delta\Vert_{L^2_x}^2
\le
C\nu^{-1} \Delta t^2,
\end{align*}
for some constant $C>0$, independent of $\Delta t$. Noting that $\tilde{r}(t=0) = 0$, we conclude that
\[
\sup_{t\in [0,T]} \Vert \tilde{r}^\Delta(t) \Vert_{L^2_x}^2
\le
CT\nu^{-1} \Delta t^2
\to 0,
\]
as $\Delta t\to 0$. This shows that
\[
\lim_{\Delta t\to 0} \Vert \beta - {\beta}^\Delta \Vert_{L^\infty_t L^2_x} = 0,
\]
i.e. the solution computed by operator splitting converges to the exact solution.
\end{proof}

%\bibliography{IffCritEnBalwForcing.bib}{}
\bibliography{bibliography.bib}{}
\bibliographystyle{plain}

\end{document}